\pdfoutput=1
\RequirePackage{ifpdf}
\ifpdf 
\documentclass[pdftex]{sigma}
\else
\documentclass{sigma}
\fi

\usepackage[mathscr]{euscript}
\usepackage[all]{xy}

\numberwithin{equation}{section}

\newtheorem{Theorem}{Theorem}[section]
\newtheorem*{Theorem3-18}{Theorem~\ref{th:Gau_is_adj}}
\newtheorem*{Theorem3-20}{Theorem~\ref{th:Gal_is_DAG}}
\newtheorem*{Theorem3-23}{Theorem~\ref{th:main}}
\newtheorem{Corollary}[Theorem]{Corollary}
\newtheorem{Lemma}[Theorem]{Lemma}
\newtheorem{Proposition}[Theorem]{Proposition}
 { \theoremstyle{definition}
\newtheorem{Definition}[Theorem]{Definition}
\newtheorem{Example}[Theorem]{Example}
\newtheorem{Remark}[Theorem]{Remark} }

\begin{document}

\allowdisplaybreaks

\newcommand{\arXivNumber}{1810.08566}

\renewcommand{\thefootnote}{}

\renewcommand{\PaperNumber}{055}

\FirstPageHeading

\ShortArticleName{Differential Galois Theory and Isomonodromic Deformations}

\ArticleName{Differential Galois Theory\\ and Isomonodromic Deformations\footnote{This paper is a~contribution to the Special Issue on Algebraic Methods in Dynamical Systems. The full collection is available at \href{https://www.emis.de/journals/SIGMA/AMDS2018.html}{https://www.emis.de/journals/SIGMA/AMDS2018.html}}}

\Author{David BL\'AZQUEZ SANZ~$^\dag$, Guy CASALE~$^\ddag$ and Juan Sebasti\'an D\'IAZ ARBOLEDA~$^\dag$}

\AuthorNameForHeading{D.~Bl\'azquez Sanz, G.~Casale and J.S.~D\'iaz Arboleda}

\Address{$^\dag$~Universidad Nacional de Colombia, Sede Medell\'in, Facultad de Ciencias,\\
\hphantom{$^\dag$}~Escuela de Matem\'aticas, Calle 59A No. 63~-~20, Medell\'in, Antioquia, Colombia}
\EmailD{\href{mailto:dblazquezs@unal.edu.co}{dblazquezs@unal.edu.co}, \href{mailto:jsdiaz@unal.edu.co}{jsdiaz@unal.edu.co}}

\Address{$^\ddag$~IRMAR, Universit\'e de Rennes 1, Campus de Beaulieu, b\^at. 22-23,\\
\hphantom{$^{\ddag}$}~263 avenue du G\'en\'eral Leclerc, CS 74205, 35042 RENNES Cedex, France}
\EmailD{\href{mailto:guy.casale@univ-rennes1.fr}{guy.casale@univ-rennes1.fr}}

\ArticleDates{Received November 14, 2018, in final form July 29, 2019; Published online August 05, 2019}

\Abstract{We present a geometric setting for the differential Galois theory of $G$-invariant connections with parameters. As an application of some classical results on differential algebraic groups and Lie algebra bundles, we see that the Galois group of a connection with parameters with simple structural group $G$ is determined by its isomonodromic deformations. This allows us to compute the Galois groups with parameters of the general Fuchsian special linear system and of Gauss hypergeometric equation.}


\Classification{53C05; 14L30; 12H05}

\renewcommand{\thefootnote}{\arabic{footnote}}
\setcounter{footnote}{0}


\section{Introduction}

Differential Galois theory deals with some algebraic properties of differential equations. Its first version, dealing with systems of linear homogeneous ordinary differential equations, was proposed by E.~Picard and E.~Vessiot at the end of the 19th century following the ideas of E.~Galois on polynomial equations. The Galois group encodes the algebraic dependence relations between the coefficients of a fundamental matrix of solutions. There is a long history of generalizations, extensions, and different frameworks for such a theory.

In \cite{landesman2008}, P.~Landesman proposes a generalized differential Galois theory for differential field extensions that extends the theory of E.~Kolchin \cite{kolchin1973} and can be applied to differential equations depending on parameters. In~\cite{cassidy2005galois}, P.~Cassidy and M.F.~Singer clarify the application of such generalized differential Galois theory to systems of linear differential equations,
\begin{gather}\label{eq:linear_intro}
\frac{\partial U}{\partial x} = A(x,s)U.
\end{gather}
Their proposal is known as parameterized Picard--Vessiot theory. It is a theory about differential field extensions. A difficulty for its application is that the field of constants must be differentially closed. This is not usual for the fields of functions that appear in the realm of algebraic geometry or complex analysis. It has been shown that, under some assumptions, this condition can be weakened by algebraic means as it is shown in \cite{gillet2013,leon2016,wibmer2012}.

The new aspect of the parameterized Picard--Vessiot theory, compared to the classical one, is that the Galois group does not only encode algebraic dependence relations between the coefficients of $U$, but differential algebraic dependence relations as functions of the parameters. Computation of the parameterized differential Galois groups is a delicate matter, many articles have been devoted to the question among them \cite{arreche2016,dreyfus2014, hardouin2017, minchenko2014, minchenko2015}.

An evidence of the geometric importance of the parameterized theory is its connection with isomonodromic families which is announced in \cite{cassidy2005galois} and further explored in \cite{gorchinskiy2014}. Namely, a~differential equation with parameters is isomonodromic if and only if its differential Galois group is conjugated to an algebraic group with constant entries. Here, based on the ideas of B.~Malgrange~\cite{malgrange2010pseudogroupes} and P.~Cartier~\cite{cartier2008groupoides} we present here another framework for the parameterized Picard--Vessiot theory that can be applied to the more usual situation of parametric families of $G$-invariant connections on algebraic varieties.
Our proposal is not different from that of Cassidy, Landesman and Singer, but its geometrical incarnation when the differential field in question is the field of rational functions on an algebraic variety. It is a geometric theory, based in the prolongation of connections to the jet bundles and reduction of the structural group. The Galois groups are going to be differential algebraic groups. Following~\cite{le2010algebraic}, we see these differential algebraic groups as systems of differential equations on the sections of an algebraic group bundle, compatible with the group operation. The main objective of this work is to unveil the relationship between the parameterized Picard--Vessiot theory and isomonodromic deformations of linear differential equations. This must be done in the context of complex geometry.

Note that \cite{arreche2016, dreyfus2014} give an algorithm to compute the Galois group of a second order differential equation. However we bypass this complex algorithm and use a geometrical argument that simplifies the exposition. It should be interesting to compare Proposition~13 from~\cite{arreche2016} to Lemma~\ref{lem}.

\subsection{Summary}

\subsubsection{Differential Galois theory of a principal connection}

In Section~\ref{sec:2} we give an exposition of the Galois theory for linear differential equations in a~geo\-metrical setting. Here, we interpret algebraic differential equations as singular foliations in varieties. We consider a principal $G$-bundle $\pi\colon P\to B$, a singular foliation $\mathscr F$ in $B$, and a $G$-invariant flat partial $\mathscr F$-connection $\mathscr D$ in $P$. This setting is based in groupoids instead of groups; it covers the case in which the field of rational first integrals (constants) $\mathbb C(B)^{\mathscr F}$ is bigger than~$\mathbb C$ and therefore not algebraically closed. In the particular case of equation~\eqref{eq:linear_intro} we have:
\begin{itemize}\itemsep=0pt
\item [(a)] the basis is $B= \mathbb C^2$ with coordinates $x$ and $s$; the base field $\mathbb C(B)$ is $\mathbb C(x,s)$;
\item [(b)] the foliation $\mathscr F$ is defined by the equation ${\rm d}s = 0$; the field of constants is $\mathbb C(B)^{\mathscr F} = \mathbb C(s)$;
\item [(c)] the principal ${\rm GL}_n(\mathbb C)$-bundle is $P = {\rm GL}_n(\mathbb C)\times B\to B$;
\item [(d)] the ${\rm GL}_n(\mathbb C)$-invariant $\mathscr F$-connection $\mathscr D$ is defined by the system:
\begin{gather*}{\rm d}s = 0, \qquad {\rm d}U = A(x,s)U\,{\rm d}x.\end{gather*}
\end{itemize}
From an algebraic point of view, we attach to the connection $\mathscr D$ the field $\mathbb C(P)^{\mathscr D}$ of its rational first integrals; rational functions in $P$ constant along leaves of $\mathscr D$. On the other hand, from a geometric point of view, we have the groupoid ${\rm Iso}(P)\to B^2$ of $G$-equivariant isomorphisms between fibers of $P\to B$. The diagonal part of ${\rm Iso}(P)$ is the group bundle ${\rm Gau}(P)\to B$ of $G$-equivariant automorphisms of the fibers of $P\to B$. This group bundle can also be constructed as $P\times_{\rm Adj}G$, the bundle associated to the adjoint action of $G$. In the particular case of equation~\eqref{eq:linear_intro} we have ${\rm Iso}(P) = {\rm GL}_n(\mathbb C)\times B^2$ and ${\rm Gau}(P) = {\rm GL}_n(\mathbb C)\times B$.

The connection $\mathscr D$ induces a singular foliation in ${\rm Iso}(P)$ that we denote as $\mathscr D\boxtimes \mathscr D$ (Remark~\ref{rm:DxD}). The restriction of such foliation to ${\rm Gau}(P)$ is a \emph{group connection} (see Section~\ref{ss:GC}) that can also be seen as the associated connection ${\rm Adj}(\mathscr D)$ induced by $\mathscr D$ in ${\rm Gau}(P)\simeq P\times_{\rm Adj}G$. In the particular case of equation~\eqref{eq:linear_intro} the foliation $\mathscr D\boxtimes \mathscr D$ is that corresponding to the system of equations
\begin{gather*}\frac{\partial R}{\partial x_1} = A(x_1,s_1)R, \qquad \frac{\partial R}{\partial x_0}= - RA(x_0,s_0);\end{gather*}
an its restriction to the diagonal part, ${\rm Gau}(P)$ is the Lax equation
\begin{gather}\label{eq:introLax}
\frac{\partial R}{\partial x} = [A(x,s),R].
\end{gather}

In order to relate the algebraic and geometric approaches, we introduce the concept of \emph{rational groupoid} (see Section~\ref{SectionRGGI}) of gauge transformations: these are Zariski-closed subsets of ${\rm Iso}(P)$ that are groupoids when restricted to a suitable open subset in $B$. Any subfield of $\mathbb C(P)$ is stabilized by a rational groupoid of gauge transformations, and conversely any rational groupoid of gauge stabilized a field of rational invariants. We state (Proposition~\ref{Galois_correspondence_2}) that this relation between groupoids and their invariants is in fact a bijective Galois correspondence
\begin{gather*}
\begin{Bmatrix}
\mbox{rational subgroupoids}
\\
\mbox{ of } \mathrm {Iso}(P)
\end{Bmatrix}
\leftrightarrow
\begin{Bmatrix}
G\mbox{-invariant subfields of } \mathbb C(P)
\\ \
\mbox{containing } \mathbb C
\end{Bmatrix}.
\end{gather*}
The \emph{Galois groupoid} ${\rm GGal}(\mathscr D)\subset {\rm Iso}(P)$ is defined as the rational groupoid stabilizing $\mathbb C(P)^{\mathscr D}$. The (intrinsic) Galois group ${\rm Gal}(\mathscr D)\subset {\rm Gau}(\mathscr D)$ is seen as bundle in groups over the base $B$ and it is obtained as the diagonal part of the groupoid (Definition~\ref{df:GGal_classical}). In the case without rational first integrals, $\mathbb C(B)^{\mathscr F} = \mathbb C$, the classical differential Galois group can be recovered as the fiber of~${\rm Gal}(\mathscr D)$.

The Galois groupoid ${\rm GGal}(\mathscr D)$ is also the smallest rational groupoid of gauge transformations which is tangent to the foliation $\mathscr D\boxtimes \mathscr D$ (Remark~\ref{rm:topGalois}). In the particular case of equation~\eqref{eq:linear_intro} it is the Zariski-closure of the solution of the following Cauchy problem
\begin{gather}\label{eq:cauchy_intro}
\frac{\partial R}{\partial x_1} = A(x_1,s_1)R, \qquad \frac{\partial R}{\partial x_0} = - RA(x_0,s_0), \qquad R(x_0,s_0,x_1,s_1) = {\rm Id}, \qquad s_0 = s_1.
\end{gather}

\subsubsection{The parameterized Galois groupoid of a connection}

Section~\ref{sec:3} is devoted to linear differential Galois theory with parameters. Readers that are not familiar with the language of differential algebraic groups as in \cite{le2010algebraic} should read Appendix~\ref{ap:DAG} before this section. Here we fix a submersion $\rho\colon B\to S$ and we take $\mathscr F = \ker({\rm d}\rho)$, the foliation tangent to the fibers of~$\rho$. A $G$-invariant flat $\mathscr F$-connection is an algebraic family of $G$-invariant flat connections parameterized by $S$. In the particular case of equation~\eqref{eq:linear_intro} the submersion is $\mathbb C^2\to \mathbb C$, $(x,s)\mapsto s$.

We consider $J(P/B)\to B$ the bundle of $k$-jets of sections of $P\to B$ which is a pro-algebraic bundle. Inside $J(P/B)$ there is a pro-algebraic subvariety $\Gamma(\mathscr D)$ consisting of formal $\mathscr D$-horizontal sections of $P\to B$. In the particular case of equation~\eqref{eq:linear_intro} the equations of $\Gamma(\mathscr D)\subseteq J(P/B)$ are obtained by successive derivation of the equation~\eqref{eq:linear_intro}:
\begin{gather*}
U_x = A(x,s)U, \\
U_{xx} = A_x(x,s)U + A(x,s)U_x, \\
U_{xs} = A_{s}(x,s)U + A(x,s)U_{ss},\\
U_{xxx} = A_{xx}(x,s)U + 2A_x(x,s)U_x + A(x,s)U_{xx}, \\
U_{xxs} = A_{xs}(x,s)U + A_x(x,s)U_s + A_s(x,s)U_x + A(x,s)U_{xs},\\
\hspace*{9.4mm} \vdots
\end{gather*}
Note that this equation allow to write down all derivatives of $U$ with respect to $x$ in terms of derivatives of $U$ with respect to~$s$. In this case we have that $\Gamma(\mathscr D)$ is coordinated by $x,s,U,U_s,U_{ss},U_{sss},\ldots$ and thus its field of rational functions is $\mathbb C(\Gamma(\mathscr D)) = \mathbb C(x,s,U,U_s,U_{ss},\allowbreak U_{sss},\ldots)$.
An element $j_{(x_0,s_0)}\Phi(x,s) \in \Gamma(\mathscr D)$ is the formal solution of~\eqref{eq:linear_intro} with initial conditions around~$x_0$,~$s_0$:
\begin{gather*}\Phi(x_0,s) = U_0 + (s-s_0)U_1 + \frac{(s-s_0)^2}{2}U_2 + \frac{(s-s_0)^3}{6}U_3 + \cdots.\end{gather*}
Quantities $U_0,U_1,U_2,{\ldots}$ are the values of the coordinate functions $U,U_s,U_{ss},{\ldots}$ at $j_{(x_0,s_0)}\!\Phi(x,s)$.

The bundle $\Gamma(\mathscr D)\to B$ is endowed (after choice of a moving frame in $S$, Proposition~\ref{pr:principalGD}) of a~structure of principal $\hat\tau_{m,}G$-bundle, where $m$ is the dimension of $S$ and $\hat\tau_{m}G$ is the pro-algebraic group of formal maps from $(\mathbb C^m, 0)$ to~$G$. In the particular of equation~\eqref{eq:linear_intro} the elements of~$\hat\tau_1{\rm GL}_n(\mathbb C)$ are power series
\begin{gather*}j_0\sigma(s) = \sigma_0 + s\sigma_1 + \frac{s^2}{2}\sigma_2+ \frac{s^3}{6}\sigma_3 + \cdots\end{gather*}
of $n\times n$ matrices with non-degenerate initial term $\sigma_0$. The action of $\hat\tau_1{\rm GL}_n(\mathbb C)$ in $\Gamma(\mathscr D)$ is given by left multiplication and Taylor expansion
\begin{gather*}(j_0\sigma(s)) (j_{(x_0,s_0)}\Phi(x,s)) = j_{(x_0,s_0)} (\sigma(s-s_0)\Phi(x,s)).\end{gather*}

We define, by means of the total derivative operators, the prolongation $\mathscr D^{(\infty)}$ of $\mathscr D$ to $\Gamma(\mathscr D)$ (Definition \ref{df:prolD}, for each finite order $k$). It is a $\hat\tau_mG$-invariant $\mathscr F$-connection in $\Gamma(\mathscr D)\to B$. In the particular case of equation \eqref{eq:linear_intro} the prolongation corresponds to higher differentiation of the original equation with respect to the parameter:
\begin{gather*}
{\rm d}s =0, \\
{\rm d} U = A(x,s)U\,{\rm d}x, \\
{\rm d}U_s = (A_s(x,s)U + A(x,s)U_s)\,{\rm d}x, \\
{\rm d}U_{ss} = (A_{ss}(x,s)U + 2A_s(x,s)U_s + A(x,s)U_{ss})\,{\rm d}x, \\
\hspace*{9.6mm} \vdots
\end{gather*}
We define the field of \emph{rational differential invariants} of the connection $\mathscr D$ as the field $\mathbb C(\Gamma(\mathscr D))^{\mathscr D^{(\infty )}}$ of rational first integrals of the connection $\mathscr D^{(\infty)}$ (see Section~\ref{ss:PGal}). After truncation to arbitrary finite order we apply the classical differential Galois theory of Section~\ref{sec:2} and obtain a Galois groupoid ${\rm PGal}_{\infty}(\mathscr D) \subset {\rm Iso}(\Gamma(\mathscr D))$ (Definition~\ref{def:GaloisgroupP}). This is the pro-algebraic rational groupoid of gauge transformations of the bundle of jets of $\mathscr D$-horizontal sections that stabilize all rational differential invariants of~$\mathscr D$.

Its diagonal part ${\rm Gal}_{\infty}(\mathscr D)$ is a rational pro-algebraic group sub-bundle of ${\rm Gau}(\Gamma (\mathscr D))$ which is itself a pro-algebraic group sub-bundle of ${\rm Gau}(J(P/B))$. From Lemma~\ref{lm:JG} we know a canonical isomorphism ${\rm Gau}(J(P/B)) \simeq J({\rm Gau}(P)/B)$. That is, automorphisms of the fibers of $J(P/B)$ can be seen as jets of gauge transformations of~$P$. Note that $J({\rm Gau}(P)/B)$ is a jet bundle over~$B$, and therefore it is endowed of a differential structure: total derivative operators, and so on. We realized that the group bundle of gauge automorphisms of~$\Gamma(\mathscr D)$ is defined by some differential equations in ${\rm Gau}(P)$. Thus, it is a~\emph{differential algebraic subgroup}. Moreover, so it is the Galois group.

\begin{Theorem3-18}
The group bundle of gauge automorphisms $\Gamma(\mathscr D)$ is identified with the diffe\-ren\-tial algebraic subgroup of gauge symmetries of $\mathscr D$, defined by the group connection ${\rm Adj}(\mathscr D)$,
\begin{gather*}{\mathrm{Gau}}(\Gamma (\mathscr D))\simeq \Gamma({\rm Adj}(\mathscr D)))\subset J({\mathrm{Gau}}(P)/B).\end{gather*}
\end{Theorem3-18}

\begin{Theorem3-20}
The parameterized Galois group bundle ${\mathrm{Gal}}_{\infty}(\mathcal D)$ is a differential algebraic subgroup of $J({\mathrm{Gau}}(P)/B)$.
\end{Theorem3-20}

Here $\Gamma({\rm Adj}(\mathscr D))\subset J({\mathrm{Gau}}(P)/B)$ is the differential algebraic group of gauge symmetries of the connection $\mathscr D$, which in the particular case of equation~\eqref{eq:linear_intro} consists of formal solutions of the Lax equation~\eqref{eq:introLax}.
The parameterized Galois group is a differential algebraic subgroup of it.

\subsubsection{Isomonodromic deformations}

In our geometric framework we interpret isomonodromic deformations as extensions of vector field distributions. In Section~\ref{ss:isom} we define:

An \emph{isomonodromic deformation} of the connection $\mathscr D$ with parameters in $S$ is a partial $G$-invariant connection $\tilde{\mathscr D}$ extending~$\mathscr D$.
If we take $\tilde{\mathscr F} = \pi_*\big(\tilde{\mathscr D}\big)$, then $\tilde{\mathscr D}$ is a partial $\tilde{\mathscr F}$-connection. Paths along the leaves of $\rho_*(\tilde{\mathscr F})$ are called \emph{isomonodromic paths}.

Under the hypothesis of simple structural group, as a direct consequence of Kiso and Cassidy results \cite{cassidy1989classification, kiso1979local}, isomonodromic deformations provide the equations for the Galois group with parameters.

\begin{Theorem3-23}
Let $P\to B$ be an affine principal bundle with simple group $G$, $\rho\colon B\to S$ a~dominant map, $\mathscr F = \ker({\rm d}\rho)$, and $\mathscr D$ a principal $G$-invariant connection with parameters in~$S$. Let us assume that the Galois group of $\mathscr D$ is ${\mathrm{Gau}}(P)$. Then, there is a biggest isomonodromic deformation $\tilde{\mathscr D}$ of $\mathscr D$ and the Galois group with parameters ${\mathrm{Gal}}_{\infty}(\mathscr D)$ is the group $\Gamma\big({\mathrm{ Adj}}\big(\tilde{\mathscr D}\big)\big)$ of gauge symmetries of~$\tilde{\mathcal D}$.
\end{Theorem3-23}

In Example~\ref{ex:fuchsian} we consider the bundle $P\to B$ where $B = \mathbb C^{k+1} \times \mathfrak{sl}_n(\mathbb C)^k$ with coordinates $x,a_1,\ldots,a_k, A_1,\ldots,A_k$; $S= \mathbb C^k\times \mathfrak{sl}_n(\mathbb C)^k$ and $P = {\mathrm{ SL}}_n(\mathbb C) \times B$. We also consider the projection $\rho\colon B\to S$. In this setting we study the general Fuchsian system with $k$ singularities and trace free matrices, with $ \mathrm{SL}_n(\mathbb C)$-connection given by
\begin{gather*}
\frac{\partial U}{\partial x} = \left(\sum_{i=1}^n \frac{A_i}{x-a_i}\right)U.
\end{gather*}
Its isomonodromic deformations are well known and given by the Schlessinger system. Theorem~\ref{th:main} allows us to describe its differential Galois group as the group of gauge symmetries of the Schlessinger system. It is defined by the differential equations
\begin{gather*}
\frac{\partial \sigma}{\partial x} = \left[\sum_{i=1}^n \frac{A_i}{x-a_i},\sigma \right], \qquad
\frac{\partial \sigma}{\partial a_i} = \left[\frac{A_i}{x-a_i}, \sigma \right].
\end{gather*}

Section \ref{s:GHyp} is devoted to Gauss hypergeometric equation
\begin{gather*}\label{HG_intro}
\quad x(1-x)\frac{{\rm d}^2u}{{\rm d}x^2} + \{\gamma -(\alpha+\beta+1) x \} \frac{{\rm d}u}{{\rm d}x} - \alpha\beta u=0,
\end{gather*}
which is seen as a linear connection with parameters, in the bundle $P = {\mathrm{ GL}}_2(\mathbb C) \times B \to B$, where $B = \mathbb C_{x}\times S$ and $S = \mathbb C^{3}_{\alpha,\beta,\gamma}$. Then by a series of arguments, that include reducing the equation to $ \mathrm{PGL}_2 (\mathbb C) \times B$, the absence of isomonodromic deformations for the reduced equation and also the projection given by the determinant, we arrive at the following description of its the Galois group:

The differential Galois group with parameters of Gauss' hypergeometric equation is given by
\begin{gather*}
\label{eq:GrupoH_intro}
{\mathrm{Gal}}_{\infty}(\mathcal H) = \{j\sigma\in {\mathrm{Gau}}(J(\mathcal H)) \,|\, j \det (\sigma)\in {\mathrm{Gal}}_{\infty}(\mathcal H_{\det})\},
\end{gather*}
which is defined by the differential equations
\begin{gather*} \partial_x\sigma = \left[\left(\begin{matrix}
0 & 1 \\
\dfrac{\alpha\beta}{x(1-x)} & \dfrac{(\alpha+\beta+1)x - \gamma}{x(1-x)}
\end{matrix} \right)
,\sigma\right],\\
\partial_x\det(\sigma) = \partial_c\det(\sigma) = \partial_a\left(\frac{\partial_a\det(\sigma)}{\det(\sigma)}\right)=
\partial_b\left(\frac{\partial_a\det(\sigma)}{\det(\sigma)}\right)=
\partial_b\left(\frac{\partial_b\det(\sigma)}{\det(\sigma)}\right)= 0,
\end{gather*}
where $\sigma$ is an invertible $2\times 2$ matrix depending on $x$, $\alpha$, $\beta$, $\gamma$ and the differential operators $\partial_a$, $\partial_b$, $\partial_c$ are those given in Proposition~\ref{pr:abc}.

A consequence of the computation of ${\rm Gal}_{\infty}(\mathcal H)$ is the following: let us consider the Gauss hypergeometric series
\begin{gather*}{}_{2}{\rm F}_1(\alpha,\beta,\gamma;x) = \sum_{n=0}^\infty \frac{(\alpha)_n(\beta)_n}{(\gamma)_n}\frac{x^n}{n!}.\end{gather*}
It is well known that
\begin{gather*}{}_2{\bf F}_1 = \left(
\begin{matrix}
{}_2{\rm F}_1(\alpha,\beta,\gamma;x) & x^{1-\gamma}{}_2{\rm F}_1(1+\alpha-\gamma,1+\beta-\gamma,2-\gamma;x) \\
\dfrac{{\rm d}}{{\rm d}x}{}_2{\rm F}_1(\alpha,\beta,\gamma;x) & \dfrac{{\rm d}}{{\rm d}x}({}_2{\rm F}_1(1+\alpha-\gamma,1+\beta-\gamma,2-\gamma;x))
\end{matrix}
\right)\end{gather*}
is a fundamental matrix of solutions of \eqref{HGmatrix} and thus $\mathcal H$-horizontal section. The only differential equation satisfied by ${}_{2}{\bf F}_1$ with respect to the parameters $\alpha$, $\beta$, $\gamma$ are those defining its Galois group, that are differential equations in $\det({}_2{\bf F}_1)$. Therefore we may state that the \emph{hypergeometric series ${}_2{\rm F}_1(\alpha,\beta,\gamma; x)$ does not satisfy any algebraic partial differential equation with respect to the parameters~$\alpha$,~$\beta$,~$\gamma$}.

Appendix~\ref{ap:DAG} is an exposition of the theory of differential algebraic groups, following~\cite{le2010algebraic}, from a geometric point of view. Finally, in Appendix~\ref{appendix-B} we explore the relation between the algebraic theory of~\cite{cassidy2005galois} and ours.

\section{Differential Galois groupoid of a principal connection}\label{sec:2}

\subsection{On rational equivalence relations}

Let $P$ be a complex irreducible algebraic affine smooth variety. We say that a Zariski-closed subset $R\subset P\times P$ is \emph{a rational equivalence relation} if:
\begin{enumerate}\itemsep=0pt
\item[1)] there is an open subset $P'$ such that $R|_{P'}:= R \cap (P'\times P')$ is a equivalence relation in $P'$;
\end{enumerate}
 and it satisfies any of the following equivalent conditions:
\begin{enumerate}\itemsep=0pt
\item[2)] for any open subset $U$ of $P$, $R$ is the Zariski closure of $R|_U$;
\item[3)] irreducible components of $R$ dominate $P$ by projections.
\end{enumerate}

\begin{Remark}Meromorphic equivalence relations, and generic quotients by them, have been already studied in~\cite{grauert1986meromorphic}. Generic quotients in a broader sense have been studied in SGA3~\cite{demazure1970schemas}. Galois correspondences shown here (Propositions~\ref{Galois_correspondence_1} and~\ref{Galois_correspondence_2}) can be seen as a direct consequence of results shown in SGA3 \cite[Th\'eor\`eme~8.1]{demazure1970schemas}.
\end{Remark}

\begin{Remark}Let us recall that the diagonal $\Delta_P\subset P\times P$ is included in any rational equivalence relation, and it dominates $P$ by projections. Therefore, if $R\subset P\times P$ is an irreducible Zariski-closed subset then it is a rational equivalence relation if and only if it satisfies condition~(1).
\end{Remark}

The following result is a particular case of \cite[Th\'eor\`eme~1.1]{malgrange2001}.

\begin{Theorem}Let $R$ be a rational equivalence relation in $P$. There is an open subset $P'\subset P$ such that $R|_{P'}$ is a smooth equivalence relation in $P'$.
\end{Theorem}

\begin{Example}Let $P$ be $\mathbb C^2$ with coordinates $x$, $y$; we consider coordinates $x_1$, $y_1$, $x_2$, $y_2$ in~$P\times P$. Then, the closed subset~$R$, defined by the equation \begin{gather*}x_1y_2 - x_2y_1 = 0,\end{gather*} is rational equivalence relation in~$P$.
\end{Example}

Rational equivalence relations are reflexive and symmetric. They are not in general transitive (as relations), but generically transitive, i.e., they become transitive after restriction to a dense open subset. For $p\in P$ let us denote $[p]_R$ its \emph{class}, the set of elements that are related by~$R$ to~$p$, i.e., $\pi_1\big(\pi_2^{-1}(\{p\})\big)$. It is a Zariski closed subset of~$P$. For generic $p\in P$ we have $\dim ([p]_R) = \dim (R) - \dim (P)$. Thus, we say that a rational equivalence relation is \emph{finite} if its dimension equals the dimension of~$P$; it implies that generic equivalence classes are finite.

\subsubsection{Irreducible rational equivalence relations}

Let us proceed to the description of all irreducible rational equivalence relations in $P$. Let us consider $\mathscr F$ a singular foliation in $P$,
and let $P'\subset P$ be the complement of ${\rm sing}(\mathscr F)$. We say that~$\mathscr F$ is a \emph{foliation with algebraic leaves} if for all $p\in P'$ the leaf of $\mathscr F$ that passes through $p$ is open in its Zariski closure. A result of G\'omez-Mont \cite[Theorem~3]{gomez1989integrals} ensures that there is an algebraic variety $V$ and a rational dominant map $f\colon P\dasharrow V$ such that the closure of a generic $f$-fibre is the closure of a leaf of~$\mathscr F$.

The foliation $\mathscr F$ induces an rational equivalence relation in $P$,
\begin{gather*}R_{\mathscr F} =\overline{\{(p,q)\in {\rm dom}(f)^2 \,|\, f(p) = f(q)\}}.\end{gather*}
We say that the rational equivalence relation $R_{\mathscr F}$ is the relation ``to be on the same leaf of $\mathscr F$''. For generic $p\in P$ the class $[p]_{R_{\mathscr F}}$ is the Zariski closure of the leaf passing through $p$.

\begin{Lemma}\label{Gomez-Mont_as_ROE}Let $R$ be a rational irreducible equivalence relation in $P$. Then, there is a singular foliation $\mathscr F$ with algebraic leaves in $P$ such that, $R$ is the relation $R_{\mathscr F}$ ``to be on the same leaf of~$\mathscr F$''.
\end{Lemma}

\begin{proof}Let us consider $R$ an irreducible rational equivalence relation in $P$. For each $p\in P$ we define $\mathscr F_p = T_{p}([p]_R)$. This $\mathscr F$ is a singular foliation of rank $\dim(R)-\dim(P)$. For generic $p\in P$ the leaf of $\mathscr F$ that passes through $p$ is the regular part of the irreducible component of $[p]_R$ that contains~$p$. Then $\mathscr F$ is a foliation with algebraic leaves. We consider now an algebraic variety~$V$ of dimension $\operatorname{codim}(R)$ and a rational dominant map $f\colon P\dasharrow V$ such that for generic~$p$ in~$P$ the leaf that passes through~$p$ is Zariski dense in $\overline{f^{-1}\{f(p)\}}$. By definition $R_{\mathscr F}$ is a rational equivalence relation of the same dimension of~$R$, and contained in~$R$. From the irreducibility of~$R$ we conclude $R=R_{\mathscr F}$.
\end{proof}

\begin{Remark}In the open subset in which the foliation $\mathscr F$ is regular, leaves of $\mathscr F$ are irreducible. A generic equivalence class is the Zariski closure of a generic leaf of $\mathscr F$ thus it is irreducible.
\end{Remark}

In the previous discussion, the algebraic variety $V$ is defined up to birational equivalence. The field $f^*(\mathbb C(V)) = \mathbb C(P)^{\mathscr F}$ of rational first integrals of $\mathscr F$ allows us to recover the rational equivalence relation
\begin{gather*}R_{\mathscr F} = \overline{\{(p,q)\in P\times P \,|\,\forall\, h\in \mathbb C(P)^{\mathscr F}\,\,
\mbox{ if }p, \ q \mbox{ are in the domain of }h\mbox{ then }h(p) = h(q) \}}.\end{gather*}

Let us recall that a subfield $\mathbb K$ of $\mathbb C(P)$ is \emph{relatively algebraically closed} if all elements of $\mathbb C(P)$ that are algebraic over $\mathbb K$ are in $\mathbb K$. It is also well known that any function which is algebraic over~$\mathbb C(P)^{\mathscr F}$ is also a first integral of $\mathscr F$. Thus, $\mathbb C(P)^{\mathscr F}$ is relatively algebraically closed in~$P$. Given a~rational equivalence relation $R$, we can always consider its field of first integrals,
\begin{gather*}\mathbb C(P)^R = \{f\in \mathbb C(P)\, | \, \forall\, (p,q)\in R \mbox{ if }p, \ q \mbox{ are in the domain of }f\mbox{ then }f(p) = f(q)\}.\end{gather*}
Reciprocally, from any relatively algebraically closed subfield $\mathbb K \subset\mathbb C(P)$ containing $\mathbb C$ we can recover a unique irreducible rational equivalence relation $R$ having $\mathbb C(P)^R = \mathbb K$. All this can be summarized in the following lemma.

\begin{Lemma}The following maps are bijective:
\begin{itemize}\itemsep=0pt
\item[$(a)$] The assignation $\mathscr F \leadsto R_{\mathscr F}$ that sends any singular foliation $\mathscr F$ with algebraic leaves to the irreducible rational equivalence relation ``to be on the same leaf of $\mathscr F$.
\item[$(b)$] The assignation $R \leadsto \mathbb C(P)^R$ that sends any irreducible rational equivalence relation to its field of rational invariants.
\item[$(c)$] The assignation that sends any relatively algebraically closed subfield $\mathbb K\subset \mathbb C(P)$ to the foliation $\mathscr F = \ker({\rm d}\mathbb K)$.
\end{itemize}
\end{Lemma}

\subsubsection{Finite rational equivalence relations} In this case $\dim R = \dim P$ and the projection $\pi_1\colon R\to P$ is a dominant map. Thus, there is an open subset $P'$ such that $\pi_1\colon R|_{P'} \to P'$ is a covering map. Note that, in order to construct this open set $P'$ we may remove not only the ramification but its orbit by the equivalence relation. Let us distinguish between two cases. First, let us assume $\pi_1|_{P'}$ is a trivial covering. We have the following result.

\begin{Lemma}Let us assume that there is an open subset $P'\subset P$ such that $\pi_1\colon R|_{P'}\to P'$ is a~trivial covering of order $d$. Then, there is free action of a~finite group of order $d$, $G\subset {\rm Aut}(P')$ such that
\begin{gather*}R = \overline{\{(p,g(p))\,|\, p\in P',\, g\in G\}}.\end{gather*}
\end{Lemma}

\begin{proof}Let $R|_{P'} = R'_0\cup R'_1 \cup\cdots \cup R'_{d-1}$ be the decomposition of $R$ in irreducible components. For each $i=0,\ldots,{d-1}$ we define $g_i = \pi_2\circ\pi_1|_{R'_i}^{-1}$. By the reflexive and transitive properties of~$R$ we have that $\{g_0,\ldots,g_{d-1}\}$ is a group of regular automorphims of~$P'$. By definition of the~$g_i$, a pair $(p,q)$ is in $R_i|_{P'}$ if and only if $q = g_i(p)$. Thus,
\begin{gather*}R|_{P'} = \{(p,g_i(p) \,|\, p\in P',\,i=0,1,\ldots,{d-1}\}.\tag*{\qed}\end{gather*}\renewcommand{\qed}{}
\end{proof}

The quotient by free actions of finite groups is well known (see, for instance~\cite{demazure1970schemas}) and then there is a rational quotient map $f\colon P\dasharrow P'/G$ of degree $d$. Note that the finite group $G$ acts on~$\mathbb C(P)$, and that this action fixes $f^{*}(\mathbb C(P/G))$ which is identified with $\mathbb C(P)^G$. We have the following

\begin{Lemma}The assignation $R \leadsto \mathbb C(P)^R$ establishes a bijective correspondence between finite rational equivalence equations that are trivial coverings of open subsets of $P$, and subfields $\mathbb K$ of~$\mathbb C(P)$ such that $\mathbb K \subset \mathbb C(P)$ is an algebraic Galois extension.
\end{Lemma}

\begin{proof}We need only to check that starting with a Galois subfield $\mathbb K \subset \mathbb C(P)$ we recover a~trivial covering. Let $G$ be the group of automorphisms of $\mathbb C(P)$ fixing $\mathbb K$. They are birational automorphisms of~$P$. We may find an open subset $P'$ invariant by those automorphisms. It is clear that this allows us to recover a trivial rational equivalence relation in~$P$.
\end{proof}

\looseness=1 Finally let us consider the second case in which $\pi_1\colon R|_{P'}\to P'$ is a covering, but not trivial. In such case there is finite covering $\rho\colon P''\to P'$ that trivializes $\pi_1$ (and $\pi_2)$. We define $R''\subset P''\times P''$ the pullback of the equivalence relation $R|_{P'}$ to $P''$. Now, $R''$ is a trivial cover of $P''$ by $\pi_1$ or $\pi_2$, and thus there is a group of automorphisms of $P''$ inducing $R''$. We have proven the following

\begin{Lemma}Let $R$ be a finite rational equivalence relation in $P$. There is an open subset $P'$, a finite covering $\rho\colon P''\to P'$ and
a finite group $G$ of regular automorphisms of $P''$ such that
\begin{gather*} R = \overline{\{ (\rho(p),\rho(g(p)))\, |\, p\in P'',\,g\in G\}}.\end{gather*}
\end{Lemma}

In this case, we have that the generic quotient $P/R$ coincides with the quotient $P''/G$ which is, up to birational equivalence, determined by its field of rational functions. We also have

\begin{Lemma}\label{Lema_REQ_finite}The assignation $R \leadsto \mathbb C(P)^R$ establishes a bijective correspondence between finite rational equivalence relations in $P$, and subfields $\mathbb K$ of $\mathbb C(P)$ such that $\mathbb K \subset \mathbb C(P)$ is an algebraic extension.
\end{Lemma}

\subsubsection{General rational equivalence relations} Finally, given a non-irreducible rational equivalence relation $R$, we may proceed to construct the quotient in two steps. First, we consider $R^0$ the irreducible component containing the diagonal~$\Delta$. In this case, the quotient is given by Lemma~\ref{Gomez-Mont_as_ROE}. Then, we project the rational equivalence relation to the quotient, obtaining a finite rational equivalence relation. Then we apply Lemma~\ref{Lema_REQ_finite}. Finally we obtain that rational equivalence relations are determined by its fields of invariants, and vice-versa. Summarizing

\begin{Proposition}[Galois correspondence for rational equivalence relations]\label{Galois_correspondence_1} The map
\begin{align*}
\{\mbox{rational equivalence relations in }P\} &\to \{\mbox{subfields of }\mathbb C(P)\mbox{ containing }\mathbb C\},\\ R &\to \mathbb C(P)^R
\end{align*}
is a bijective correspondence, and anti-isomorphism of lattices.
\end{Proposition}

\begin{Remark}
In the correspondence given by Proposition \ref{Galois_correspondence_1}:
\begin{enumerate}\itemsep=0pt
\item[1)] irreducible rational equivalence relations are given by rational foliations with algebraic leaves and correspond to relatively algebraically closed subfields of $\mathbb C(P)$;
\item[2)] finite trivial rational equivalence relations correspond to actions of finite groups on open subsets of~$P$.
\end{enumerate}
\end{Remark}

\subsection{Rational groupoids}
Let us fix $B$ a complex irreducible smooth affine algebraic variety. By an \emph{algebraic groupoid} acting on $B$ we mean a smooth algebraic complex Lie groupoid
$(\alpha,\beta)\colon \mathcal G \to B\times B$.

\begin{Definition} A Zariski-closed subset $\mathcal H\subset \mathcal G$ is called a \emph{rational subgroupoid} if it satisfies:
\begin{enumerate}\itemsep=0pt
\item[1)] there is an open subset $B'$ of $B$ such that $\mathcal H|_{B'}$ is a Lie subgroupoid of $\mathcal G|_{B'}$;
\end{enumerate}
and it satisfies any of the following equivalent conditions:
\begin{enumerate}\itemsep=0pt
\item[2)] for any open subset $B'$ of $B$, $\mathcal H = \overline{\mathcal H|_{B'}}$;
\item[3)] irreducible components of $\mathcal H$ dominate $B$ by the source (or target) map.
\end{enumerate}
\end{Definition}

\begin{Remark}If $\mathcal H$ is irreducible, it automatically satisfies~(3); it contains the identity.
\end{Remark}

It is clear that the notion of rational subgroupoid generalizes that of rational equivalence relations, the latter being the rational subgroupoids of $B\times B$.

Let $\pi_{\mathscr G}\colon \mathscr G\to B$ be an \emph{algebraic group bundle}. The fibers of $\pi_{\mathscr G}$ are complex algebraic groups where the group operation and inversion are defined as morphisms over~$B$. There is a natural dictionary between group bundles over open subsets of $B$ and algebraic groups defined over~$\mathbb C(B)$. Thus, sections of~$\mathscr G$ define an algebraic group $\tilde{\mathscr G}$. Algebraic $\mathbb C(B)$-subgroups of $\tilde{\mathscr G}$ are in bijective correspondence with \emph{rational group subbundles} of $\mathscr G$. That is, Zariski closed subsets $\mathscr H\subset \mathscr G$ such that:
\begin{enumerate}\itemsep=0pt
\item[1)] there is an open subset $B'$ such that $\pi_{\mathscr G}\colon\mathscr H|_{B'}\to B'$ is an algebraic group bundle;
\item[2)] irreducible components of $\mathscr H$ dominate $B$ by $\pi_{\mathscr G}$.
\end{enumerate}

Let us consider the restriction of the algebraic groupoid $\mathcal G$ to the diagonal of~$B$,
\begin{gather*}\mathcal G^{\rm diag} = \{g\in\mathcal G \colon \alpha(g)=\beta(g)\}.\end{gather*}
It is an algebraic group bundle $\alpha\colon \mathcal G^{\rm diag}\to B$. Given an rational subgroupoid $\mathcal H\subset \mathcal G$, the intersection $\mathcal H\cap \mathcal G^{\rm diag}$ is not a rational group bundle over $B$, it may contain some irreducible components that do not dominate the basis.

\begin{Example}Let us consider $B = \mathbb C^2$, and $\mathcal G = \mathbb C^*\times B\times B$ where the composition law is given by $(\mu,x',y',x'',y'')\cdot(\lambda,x,y,x',y') =(\lambda\mu, x,y,x'',y'')$. The equation
\begin{gather*}\lambda = \frac{x'}{y'} \frac{y}{x}\end{gather*}
defines an algebraic subgroupoid outside of the curve $xy = 0$, and thus,
\begin{gather*}\mathcal H = \{(\lambda,x,y,x',y')\colon \lambda xy' - x'y = 0 \}\end{gather*}
is a rational subgroupoid of $\mathcal G$, which is in fact, irreducible. However the intersection with the diagonal group bundle,
\begin{gather*}\mathcal H \cap \mathcal G^{\rm diag} = \{(\lambda,x,y) \colon (\lambda-1)xy = 0\} \end{gather*}
has three irreducible components. Two of them, of equations $x = 0$ and $y=0$ do not dominate~$B$; the other one, of equation $\lambda = 1$ defines a rational group subbundle.
\end{Example}

Hence, in order to define the diagonal part of a rational groupoid we have to restrict our attention to a suitable affine open subset. It is clear that the following
definition does not depend on the choice of~$B'$.

\begin{Definition}\label{def_diagonalpart} Let $\mathcal H$ be a rational subgroupoid of $\mathcal G$. Let us consider an affine subset $B'\subset B$ such that $\mathcal H|_{B'}$ is an algebraic groupoid acting on $B'$. We define the diagonal part of $\mathcal H$,
\begin{gather*}\mathcal H^{\rm diag} = \overline{\mathcal H|_{B'} \cap \mathcal G^{\rm diag}}.\end{gather*}
\end{Definition}

\subsection{Rational groupoids of gauge isomorphisms}\label{SectionRGGI}

Let $G$ be a linear algebraic group, $\pi\colon P\to B$ a principal connected fiber bundle with structure group $G$. From now on, \emph{we assume that $B$ is affine and $P$ is affine over $B$}; many of our results can be stated in a more general setting.

By a \emph{rational reduction} of $\pi\colon P\to B$ we mean a Zariski closed subset $L\subset P$ such that:
\begin{enumerate}\itemsep=0pt
\item[1)] there is an open subset $B'\subset B$ such that $L|_{B'}\to B'$ is a reduction of the bundle $P|_{B'}\to B'$;
\item[2)] irreducible components of $L$ dominate $B$ by projection.
\end{enumerate}
As before, if $L$ is irreducible then condition~(1) implies condition~(2).

By ${\rm Iso}(P)$ we mean the groupoid of gauge isomorphisms of $P$; it is an algebraic groupoid acting on $B$. Elements of ${\rm Iso}(P)$ are $G$-equivariant maps between fibers of $P$. This groupoid can be constructed as an associated bundle ${\rm Iso}(P) = (P\times P)/G$. The class $[p,q]$ of a pair $(p,q)$ is identified with the unique $G$-equivariant map $\phi_{p,q}\colon P_{\pi(p)}\to P_{\pi(q)}$ satisfying $\phi_{p,q}(p) = q$.

By ${\rm Gau}(P)$ we mean the group bundle of gauge automorphisms of $P$; the diagonal part of~${\rm Iso}(P)$. It is an algebraic group bundle over $B$. Elements of ${\rm Gau}(P)$ are $G$-equivariant maps from a fiber of $P$ to itself. This group bundle can also be constructed as the associated bundle, the balanced construction $P[{\rm Adj}] = P\times_{\rm Adj} G $.

Given a rational subgroupoid $\mathcal G \subset {\rm Iso}(P)$, there is a corresponding field of rational invariants,
\begin{gather*}\mathbb C(P)^{\mathcal G} = \{f\, |\, f(p) = f(g(p)) \mbox{ for all } p\in P \mbox{ and }g\in \mathcal G\\
\hphantom{\mathbb C(P)^{\mathcal G} = \{}{} \mbox{whenever }p\in {\rm dom}(g)\mbox{ and }p, \ g(p)\mbox{ are in }{\rm dom}(f)\}.\end{gather*}

We would like to characterize rational subgroupoids of $\rm{Iso}(P)$ in terms of their invariants. We say that a subfield $\mathbb K \subset \mathbb C(P)$ is $G$-invariant if right translations in $P$ map $\mathbb K$ into itself.

\begin{Proposition}[Galois correspondence]\label{Galois_correspondence_2}
The map
\begin{align*}
\{\mbox{rational subgroupoids of \rm Iso}(P)\} & \to \{G\mbox{-invariant subfields of }\mathbb C(P)\mbox{ containing }\mathbb C\},\\
\mathcal G &\to \mathbb C(P)^{\mathcal G}
\end{align*}
is a bijective correspondence, and anti-isomorphism of lattices.
\end{Proposition}

\begin{proof}The quotient map $\pi_{P\times P}\colon P\times P\to {\rm Iso}(P)$, $(p,q)\mapsto \phi_{p,q}$ establishes a bijective correspondence between the set of subgroupoids of ${\rm Iso}(P)$ and the set of $G$-invariant equivalence relations in $P$. This induces a bijective correspondence between the set of rational subgroupoids of ${\rm Iso}(P)$ and $G$-invariant rational equivalence relations in $P$. Moreover, we have $\mathbb C(P)^{\mathcal G}=\mathbb C(P)^{\pi_{P\times P}^{-1}(\mathcal G)}$.

The field of invariants of a $G$-invariant rational equivalence relation is $G$-invariant. Reciprocally, the rational equivalence relation corresponding to a $G$-invariant field, is $G$-invariant. Thus, by Proposition~\ref{Galois_correspondence_1} we finish the proof.
\end{proof}

\begin{Remark}
Let us note that irreducible rational subgroupoids of ${\rm Iso}(P)$ correspond to $G$-invariant relatively algebraically closed subfields of $\mathbb C(P)$ containing $\mathbb C$. Hence, they correspond to $G$-invariant singular foliations in $P$ with algebraic leaves.
\end{Remark}

\subsection{Galois groupoid of a partial connection}\label{section:GGal} Here we give a geometric presentation based on rational subgroupoids of the classical Galois theory for $G$-invariant connections. The ideas of this section can be found in \cite{cartier2008groupoides}, see also \cite{malgrange2010pseudogroupes} for a~nonlinear version. Let us consider $G$ and $\pi\colon P\to B$ as in Section~\ref{SectionRGGI}. Additionally, we also consider a singular foliation $\mathscr F$ in $B$.

\begin{Definition}A rational partial flat Ereshmann connection in the direction of $\mathscr F$ in $P$ is a~$\pi$-projectable, singular foliation $\mathscr D$ in~$P$ of the same rank that $\mathscr F$ such that $\pi_*(\mathscr D) = \mathscr F$.
\end{Definition}

\emph{From now on we will write $\mathscr F$-connection, or just connection, to refer to partial flat Ereshmann connections.}

A local analytic section $u$ of $P$ is $\mathscr D$-horizontal if $\mathscr D_{u(x)} \subset d_xu(T_xB)$ for each $x$ in the domain of $u$. This means that each restriction of $u$ to a leaf $\mathscr L$ of $\mathscr F$ is $\mathscr D|_{\pi^{-1}(\mathscr L)}$-horizontal, and vice-versa. A vector field in $P$ tangent to $\mathscr D$ is also called $\mathscr D$-horizontal. Given a local analytic vector field $X$ tangent to $\mathscr F$ in $B$, there is a unique $\mathscr D$-horizontal vector field $\tilde X$ in $P$ such that $\pi_*\big(\tilde X\big) = X$. This~$\tilde X$ is called the $\mathscr D$-horizontal lift of~$X$. It is clear that if $X$ is a rational vector field then $\tilde X$ is also a rational vector field, moreover, if $X$ is regular in $B'\subset B$ and $\mathscr D$ is regular in $\pi^{-1}(B')$ then $\tilde X$ is regular in $\pi^{-1}(B')$. The $\mathscr D$-horizontal lifting is compatible with the Lie bracket $\widetilde{[X,Y]} = \big[\tilde X, \tilde Y\big]$.

We say that $\mathscr D$ is \emph{$G$-invariant} if for each $g\in G$ we have $R_{g*}(\mathscr D) = \mathscr D$. We say that $\mathscr D$ is \emph{regular} if it is a regular foliation. After replacing $B$ by a suitable affine subset, we may assume that any $G$-invariant rational $\mathscr F$-connection $\mathscr D$ is regular.

Let $\mathscr D$ be a $G$-invariant rational foliation on $P$. We consider the field $\mathbb C(P)^{\mathscr D}$ of rational first integrals of the foliation. This field is, by definition, $G$-invariant. Thus, by Proposition~\ref{Galois_correspondence_2} it corresponds to an irreducible rational subgroupoid of {\rm Iso}(P). The following definition is equivalent to that given in~\cite{Damien2016specialisation}.

\begin{Definition}\label{df:GGal_classical} The \emph{Galois groupoid} of $\mathscr D$ is the rational subgroupoid ${\rm GGal}(\mathscr D)$ of ${\rm Iso}(P)$ such that $\mathbb C(P)^{{\rm GGal}(\mathscr D)} = \mathbb C(P)^{\mathscr D}$. The \emph{Galois group bundle} is the diagonal part of the Galois groupoid, ${\rm Gal}(\mathscr D)= {\rm GGal}(\mathscr D)^{\rm diag}$. It is a rational group subbundle of ${\rm Gau}(P)$.
\end{Definition}

By definition, it is clear that ${\rm GGal}(\mathscr D)$ is irreducible in the Zariski topology. The Galois group bundle may not be irreducible.

\begin{Remark}The classical case of Picard--Vessiot theory corresponds to the case in which the foliation $\mathscr F$ has no rational first integrals, i.e., $\mathbb C(B)^{\mathscr F} = \mathbb C$. This setting has been exposed in~\cite{parallelisms2017}. The fiber of the Galois group bundle is isomorphic to the classical Galois group. The intrinsic Galois group is the algebraic $\mathbb C(B)$-group whose elements are the sections of the Galois group bundle.
\end{Remark}

\begin{Remark}The Galois group bundle admits also a structure of differential algebraic group. The induced associated connection ${\rm Adj}(\mathscr D)$ is a partial group connection (see Appendix~\ref{ap:DAG}) in~${\rm Gau}(P)$. It restricts to the group subbundle $\rm Gal(\mathscr D)$. Thus, the pair $({\rm Gal}(\mathscr D),{\rm Adj}(\mathscr D))$ can be seen as a differential algebraic group of order $1$. This differential algebraic group structure was already pointed out by Pillay in~\cite{pillay2004}.
\end{Remark}

\begin{Remark}\label{rm:DxD}As the group bundle ${\rm Gau}(P)\to B$ admits two equivalent descriptions, so the associated connection does. We may consider the distribution $\mathscr D\times \mathscr D$ in $P\times P$. This distribution is $G$-invariant under the diagonal action, thus it projects onto a distribution $\mathscr D\boxtimes \mathscr D$ in ${\rm Iso}(P)$. The intersection of this distribution with the tangent bundle $T({\rm Gau}(P))$ gives us the associated connection $\rm Adj(\mathscr D)$.
\end{Remark}

\begin{Remark}\label{rm:topGalois}
The Galois groupoid can be also constructed geometrically as the smallest rational subgroupoid of ${\rm Iso}(P)$ which is tangent to the distribution $\mathscr D\boxtimes \mathscr D$, see~\cite{Damien2016specialisation} or~\cite{cartier2008groupoides}. This is stated in the introduction when we say that the Galois groupoid of~\eqref{eq:linear_intro} is the Zariski closure of the solution of the Cauchy problem~\eqref{eq:cauchy_intro}, which is the union of leaves of $\mathscr D \boxtimes \mathscr D$ passing through the identity section. However, the definition of the Galois groupoids through the Galois correspondence (Proposition~\ref{Galois_correspondence_2}) is preferable for our purposes.
\end{Remark}

\begin{Remark}\label{rm:gaugesymmetries} By construction, the action ${\rm Gau}(P)\times_B P \to P$ maps the distribution \linebreak \mbox{${\rm Adj}(\mathscr D)\times_B \mathscr D$} onto $\mathscr D$. Therefore, ${\rm Adj}(\mathscr D)$ is the \emph{equation of gauge symmetries of~$\mathscr D$}. A gauge transformation is ${\rm Adj}(\mathscr D)$-horizontal if and only if it has the property of transforming $\mathscr D$-horizontal sections into $\mathscr D$-horizontal sections.
\end{Remark}

\subsection{Connections with parameters}\label{section:GGalS}
Let us also consider a submersion $\rho\colon B\to S$, so that we have commutative diagrams as follows
\begin{gather*}\xymatrix{P\ar[r]^-{\pi} \ar[rd]_-{\bar\pi} & B \ar[d]^-{\rho} \\ & S,}\qquad
\xymatrix{\mathbb C(P) & \mathbb C(B)
\ar[l]_-{\pi^*} \\ & \mathbb C(S), \ar[u]_-{\rho^*}
\ar[ul]^-{\bar\pi^*}}\end{gather*}
we assume that the fibers of $\bar\pi$ are irreducible. Let $\mathscr F$ be the foliation with algebraic leaves defined by ${\rm ker}(d\rho)$. By a \emph{$G$-invariant connection in $P$ with parameters in $S$} we mean a $G$-invariant $\mathcal F$-connection $\mathscr D$ in~$P$ \cite{le2010algebraic,malgrange2010pseudogroupes}.

For each $s\in S$ we denote $B_s = \rho^{-1}(s)$, $P_s = \bar\pi^{-1}(s)$, and $\pi_s = \pi|_{P_s}$. Thus, $\pi_s\colon P_s\to B_s$ is a~principal fiber bundle with structure group~$G$. The foliation~$\mathscr D$ is tangent to~$P_s$ its restriction~$\mathscr D_s$ is a $G$-invariant connection in $P_s$.

We identify $\mathbb C(S)$ with its image in $\mathbb C(B)$ by $\rho^*$ and $\mathbb C(B)$ with its image in $\mathbb C(P)$ by $\pi^*$ so that we have a sequence of extensions, $\mathbb C(S) \subset \mathbb C(B) \subset \mathbb C(P)$. Note that $\mathbb C(B)$ is the fixed field~$\mathbb C(P)^G$ of rational invariants by the action of $G$ and $\mathbb C(S)$ is the field $\mathbb C(B)^{\mathscr F}$ of rational first integrals of $\mathscr F$.

By ${\rm Iso}_S(P)$ we mean the subgroupoid of ${\rm Iso}(P)$ consisting on fiber isomorphisms that respect the projection $\rho$.

\begin{Remark}\label{Galois_correspondence_3} In terms of Proposition~\ref{Galois_correspondence_2} ${\rm Iso}_S(P)$ corresponds to the subfield $\mathbb C(S) \subset \mathbb C(P)$ and to the equivalence relation $P\times_S P$. Thus, we have a bijective correspondence between irreducible rational subgroupoids of ${\rm Iso}_S(P)$ and $G$-invariant subfields of $\mathbb C(P)$ containing $\mathbb C(S)$.
\end{Remark}

\begin{Proposition}
Let $\mathscr D$ be a $G$-invariant connection with parameters in $S$. We have ${\rm GGal}(\mathscr D)\allowbreak \subset {\rm Iso}_S(P)$.
\end{Proposition}

\begin{proof}Taking into account $\mathbb C(S) \subset \mathbb C(P)^{\mathscr D}$ we conclude using Remark \ref{Galois_correspondence_3}.
\end{proof}

\section{The parameterized Galois groupoid of a connection}\label{sec:3}

Let us consider, as in Section \ref{section:GGalS} an affine principal bundle $\pi\colon P\to B$ with structure group bundle $G\to S$, a submersion $\rho\colon B\to S$, $\mathscr F = \ker(d\rho)$, and $\mathscr D$ a $G$-invariant connection with parameters in $S$. Some of our arguments will require to restrict these objects to some affine subset of $B$. Accordingly, without loss of generality, \emph{we assume that $\mathscr F$ and $\mathscr D$ are regular.}

\subsection{Prolongation of a principal bundle}

Two germs of analytic sections of $P$ at $x\in B$ have the same jet of order $k$ if they are tangent up to the order $k$. This means, that in any analytic system of coordinates, they have the same power series development up to degree $k$. The space $J_k(P/B)$ of sections is a smooth affine algebraic variety and then $J(P/B) = \lim J_k(P/B)$ is a smooth pro-algebraic affine variety. The jet of order $k$ at $x\in B$ of the section $u$ is denoted by $j_x^k u$, and the section of $J_k(P/B)$ that assigns to each $x$ in the domain of $u$ the jet $j_x^ku$ is denoted by $j^k u$ (we omit the index $k=\infty$ for jets of infinite order).

By $\mathcal O_{J_k(P/B)}$ we understand the ring of regular functions in $J_k(P/B)$ with rational coefficients in $B$, that is,
$\mathcal O_{J_k(P/B)} = \mathbb C(B)\otimes_{\mathbb C[B]} \mathbb C[J_k(P/B)]$. Elements of $\mathcal O_{J_k(P/B)}$ are regular \emph{differential functions of order $k$ with rational coefficients} in $B$, or equivalently, regular differential functions in $J_k(P|_{B'}/B')$ for some open subset $B'\subset B$. We have a chain of ring extensions:
\begin{gather*}\mathbb C(B) \subset \mathcal O_{P/B} \subset \mathcal O_{J_1(P/B)} \subset \cdots \subset \mathcal O_{J_k(P/B)} \subset \cdots \subset \mathcal O_{J(P/B)}.
\end{gather*}

\subsubsection[Differential structure of $\mathcal O_{J(P/B)}$]{Differential structure of $\boldsymbol{\mathcal O_{J(P/B)}}$}
There exists a total derivative prolongation: if $X$ is a rational vector field in $B$ with domain of regularity $B'\subset B$ then its \emph{total derivative operator} $X^{\rm tot}$ is a derivation of the ring of differential functions, defined by the rule,
\begin{gather*}\big(X^{\rm tot} F\big)\big(j_x^{k+1}u\big) = \big(X\big(F\circ j^ku\big)\big)(x),\end{gather*}
for any differential function $F$ of order $k$; $X^{\rm tot}$ is a derivation of $\mathcal O_{J(P/B)}$. Note that the total derivative increases the order of differential functions, thus $X^{\rm tot}$ is not a rational vector field in~$J_{k}(P/B)$ for $k$ finite. We consider the differential ring~$(\mathcal O_{J(P/B)}, \mathfrak X_{\rm rat}(B))$. This is a $\mathfrak X_{\rm rat}(B)$-finitely generated differential ring over~$\mathbb C(B)$.

\begin{Remark}Here we work with a broader definition of differential field than usual (cf.~\cite{kolchin1973}). We consider a differential field to be a field endowed with an arbitrary Lie algebra of derivations. However, it is always possible to choose a finite commuting basis of $\mathfrak X_{\rm rat}(B)$ in order to recover the usual definition. For instance, let us take a transcendence basis of~$\mathbb C(B)$ over~$\mathbb C$. Then, the partial derivatives w.r.t.\ the basis elements are commuting rational vector fields spanning~$\mathfrak X_{\rm rat}(B)$. Their corresponding total derivative operators are a basis of derivations form~$\mathcal O_{J(P/B)}$.
\end{Remark}

\subsubsection[Prolongation of rational vector fields in $P$]{Prolongation of rational vector fields in $\boldsymbol{P}$}
Given a local analytic vector field $Y$ in $P$, its $k$-th prolongation $Y^{(k)}$ to $J_k(P/B)$ is obtained geometrically in the following way. Let us assume that $Y$ has a real flow $\{\sigma_t\}_{t\in \mathbb R}$ in a neighborhood of $p\in P$, let $x = \pi(p)$ and~$u$ be a local section such that $u(x)=p$. Let $\gamma_u$ be the graph of~$u$. For small $t$ we have that $\sigma_t(\Gamma_u)$ is, around~$x$, also the graph of a local section that we name~$u_t$. Then,
\begin{gather*}Y^{(k)}\big(j_x^ku\big) = \left.\frac{{\rm d}}{{\rm d}t}\right|_{t=0} j^k_{\pi(\sigma_t(p))} u_t.\end{gather*}
Thus, the flow of $Y^{(k)}$ is the projection in the jet space of the flow naturally induced by~$Y$ in the space of pointed germs of submanifolds of~$P$. It is easy to check that prolongation of vector fields is compatible with the Lie bracket $[X,Y]^{(k)} = \big[X^{(k)},Y^{(k)}\big]$. The prolongation $Y^{(k)}$ of a~regular vector field $Y$ is also regular. Thus, if $Y$ of a rational vector field in $P$ whose domain of definition contains an open subset of the form $\pi^{-1}(B')$ with $B'$ open in $B$ then $Y^{(k)}$ is a~derivation of~$\mathcal O_{J_k(P/B)}$. This prolongation is possible up to arbitrary order, so that we obtain a~derivation $Y^{(\infty)}$ of $\mathcal O_{J(P/B)}$ that respects the order of differential functions.

\subsubsection{Homogeneous structure of the fibers} \label{sss:hom_fibers}
In general, if $\mathcal G\to B$ is a group bundle, the group composition lifts to the jets of sections of~$\mathcal G$,
\begin{gather*}J_k(\mathcal G/B)\times_B J_k(\mathcal G/B) \to J_k(\mathcal G/B),\qquad \big(j_x^k g\big)\big(j^k_x h\big) = j_x^k (gh).\end{gather*}
Thus, for all $k$, $J_k(\mathcal G/B)\to B$ is an algebraic group bundle and $J(\mathcal G/B)\to B$ is a pro-algebraic group bundle.

Let us consider the trivial group bundle $G\times B \to B$; $J_k(G/B)$ will stand for $J_k(G \times B/B)$. There is a natural fibered action of this group bundle on~$J_k(P/B)$,
\begin{gather*}J_{k}(P/B) \times_B J_k(G/B), \qquad \big(j^k_x u, j^k_x g\big) \mapsto j^k_x (u\cdot g),\end{gather*}
so that, for all $x\in B$, $J_k(P/B)_x$ becomes a principal homogeneous space for the action of the algebraic group $J_k(G/B)_x$. Those actions are compatible with the truncations of order; hence, taking the limit to infinite order we have that $J(P/B)_x$ is a principal homogeneous space for the action of the pro-algebraic group~$J(G/B)_x$.

\subsection{Prolongation of a principal connection}\label{ss:prolpc}

The $G$-invariant integrable distribution $\mathscr D$ defines closed subsets $\Gamma_{k}(\mathscr D) \subset J_k(P/B)$. They consist of $k$-jets of $\mathscr D$-horizontal sections. This closed subset $\Gamma_k(\mathscr D)$ is the algebraic differential equation of order $k$ associated to $\mathscr D$. The chain of maps
\begin{gather*}\Gamma(\mathscr D) = \lim_{\leftarrow} \Gamma_k(\mathscr D) \to \cdots \to \Gamma_k(\mathscr D) \to \cdots \to \Gamma_{2}(\mathscr D) \to \Gamma_{1}(\mathscr D) \to P\end{gather*}
are surjective.

$\Gamma_k(\mathscr D)$ is defined by a radical ideal $\mathcal E_k(\mathscr D)$ of $\mathcal O_{J_{k}(P/B)}$. Indeed, from the Frobenius integrability condition (flatness of $\mathscr D$) we obtain that $\mathcal E(\mathscr D)$ is the radical $\mathfrak X_{\rm rat}(B)$-differential ideal of $\mathcal O_{J(P/B)}$ spanned by $\mathcal E_1(\mathscr D)$. For each $k$, $\mathcal E_{k}(\mathscr D)$ is the intersection of $\mathcal E(\mathscr D)$ with the ring $\mathcal O_{J_{k}(P/B)}$. The ring of \emph{differential functions} on horizontal sections is the $\mathfrak X_{\rm rat}(B)$-differential ring quotient $\mathcal O_{\Gamma(\mathscr D)} = \mathcal O_{J(P/B)}/ \mathcal E(\mathscr D)$.

The composition with $\rho$ induces a map $\rho^*\colon J_k(G/S)\to J_k(G/B)$. Therefore we also have an action defined over $S$,
\begin{gather*}J_{k}(P/B) \times_S J_k(G/S), \qquad \big(j^k_x f, j^k_x g\big) \mapsto j^k_x (f\cdot (\rho\circ g)).\end{gather*}

\begin{Proposition}\label{pr:homogeneous}The action of $J_k(G/S)_{\rho(x)}$ on $\Gamma_k(\mathscr D)_x$ is free and transitive for any $x\in B$.
\end{Proposition}

\begin{proof}Let $j^k_x u$ and $j^k_x v$ be in $\Gamma_k(\mathscr D)_x$. Representatives $u$ and $v$ can be taken as $\mathscr D$-horizontal sections defined in a neighborhood of~$x$. There is a unique section $g$ defined in this neighborhood~$G$ such that $u = v \cdot g$. Taking into account the~$u$ and $v$ are $\mathscr D$-horizontal, we have that $g$ is constant along the fibers of $\rho$, so it can be thought as a function from a neighborhood of $\rho(x)$ in~$S$ to~$G$. Finally we have, $j^k_x u = \big(j^k_x u\big) \cdot \big(j^k_{\rho(x)}g\big)$.
\end{proof}

\begin{Remark}\label{rm:homogeneous}It is important to note that the action is compatible with the truncation maps $J_{k+1}(G,S)\to J_k(G/S)$ and $\Gamma_{k+1}(\mathscr D))\to \Gamma_k(\mathscr D)$. Hence, we have that $J(G/S)_{\rho(x)}$ acts free and transitively in $\Gamma(\mathscr D)_x$.
\end{Remark}

\subsubsection[Trivialization of $J(G/S)$]{Trivialization of $\boldsymbol{J(G/S)}$}
By Proposition \ref{pr:homogeneous} we have that $\Gamma(\mathscr D)\to B$ is a bundle by principal homogeneous spaces. However, the structure group varies with the parameter in $S$. In order to have a principal bundle structure we should trivialize the group bundle $J(G/S)$.

Let $m$ be dimension of $S$ and $\tau_{m,k} G$ be the generalized tangent bundle of $G$ of rank $m$ and order $k$, i.e., the space of jets of order~$k$ of local analytic maps from $(\mathbb C^m,0)$ to $G$. It is an algebraic group, with the composition law, $\big(j_0^k g\big)\big(j_0^k h\big) = j_0^k (gh)$. The group ${\rm Aut}_k(\mathbb C^m,0)$ of jets of automorphisms act on $\tau_{m,k}G$ by group automorphisms on the right side by composition.

For this fixed $m$ we may take the projective limit in $k$ and thus we obtain the pro-algebraic group $\hat\tau_m G$ of formal maps from $\mathbb (\mathbb C^m,0)$ to $G$. In $\hat\tau_m G$ the group of formal automorphisms $\widehat{\rm Aut}(\mathbb C^m,0)$ acts by group automorphisms on the right side by composition.

By a \emph{frame} of order $k$ at $s\in S$ we mean a $k$-jet at $0\in \mathbb C^m$ of a locally invertible map from $(\mathbb C^m,0)$ to $(S,s)$. The frame bundle $R^k S \to S$ is a principal bundle, modeled over the group ${\rm Aut}_k(\mathbb C^m,0)$. Those bundles form a projective system and the projective limit $RS\to S$ is the frame bundle. It is a~pro-algebraic principal bundle modeled over the group $\widehat{\rm Aut}(\mathbb C^m,0)$ of formal analytic automophisms. A section of the frame bundle is called a \emph{moving frame}.

\begin{Proposition}\label{pr:principalGD}Any moving frame $\hat\varphi$ in $S$ induces by composition at the right side a triviali\-zation
\begin{gather*}\xymatrix{
J(G/S) \ar[rr]^-{\hat\varphi_*} \ar[rd] && \hat\tau_{m}G \times S \ar[ld] \\
& S,
}
\end{gather*}
and thus a structure of principal bundle in $\Gamma(\mathscr D)$,
\begin{gather*}\Gamma(\mathscr D)\times \hat\tau_m G \to \Gamma(\mathscr D),
\qquad (j_x u, j_0 g )\mapsto j_x u \cdot \big(j_0 g \circ \hat\varphi^{-1}_{\rho(x)} \circ j_x\rho\big). \end{gather*}
\end{Proposition}

\begin{proof} Let $\hat\varphi$ be a moving frame in $S$. Then, for $s\in S$ we have that $\hat \varphi_s$ is a formal isomorphism between $(\mathbb C^m,0)$ and $(S,s)$. Thus, it induces a pro-algebraic group isomorphism from $J(G/S)_s$ to $\hat\tau_m G$ that sends $j_sg$ to $j_0 (g\circ \hat\varphi_s)$. In this way we obtain the trivialization of the statement. Finally, by Proposition~\ref{pr:homogeneous} and Remark~\ref{rm:homogeneous} we have that the induced action of $\hat\tau_m G$ in $\Gamma(\mathscr D)$ gives the structure of a principal bundle.
\end{proof}

\begin{Remark}\label{rem:2trivializations} Let $\hat\varphi$ and $\hat\psi$ be two different moving frames in $S$. Then, there is a regular map~$\hat\gamma$ from $S$ to $\widehat{\rm Aut}(\mathbb C^m,0)$ such that $\hat\varphi(s) = \hat\psi(s)\hat\gamma(s)$. We have a commutative diagram
\begin{gather*}\xymatrix{
J(G/S) \ar[rr]^-{\hat\varphi_*} \ar[rrd] ^-{\hat\psi_*}
&& \hat\tau_{m}G \times S \ar[d]^-{\hat\gamma \times {\rm Id}_S} \\
&& \hat\tau_{m}G \times S
}
\xymatrix{(j_0g, s) \ar[d] \\ (j_0g\circ \hat\gamma(s),s). }
\end{gather*}
The group $\widehat{\rm Aut}(\mathbb C^m,0)$ acts on $\hat\tau_mG$ by group automorphisms. Thus, the two $\hat\tau_mG-$principal bundle structures induced in $\Gamma(\mathscr D)$ by $\hat\varphi$ and $\hat\psi$ are related by the fibered automorphism $\hat\gamma$ of $\hat\tau_mG\times S\to S$.
\end{Remark}

\begin{Remark}As before, our constructions are compatible with the truncation to any finite order. The moving frame $\varphi$ induces frames of any order, and thus, also induces compatible principal bundle structures in~$\Gamma_k(\mathscr D)$ modeled over $\tau_{m,k}G$ in the sense that the following diagram
\begin{gather*}\xymatrix{
\Gamma(\mathscr D)\times \hat\tau_m G \ar[r]\ar[d]& \Gamma(\mathscr D) \ar[d] \\
\Gamma_k(\mathscr D)\times \tau_{m,k} G \ar[r] & \Gamma_k(\mathscr D) }
\end{gather*}
is commutative.
\end{Remark}

\begin{Remark}\label{rm:isoS}For $k>\ell$ there are natural surjective morphisms ${\rm Iso}(\Gamma_{k}(\mathscr D))\to {\rm Iso}(\Gamma_{\ell}(\mathscr D))$. The projective limit of this system yields ${\rm Iso}(\Gamma(\mathscr D))$ which is a pro-algebraic Lie groupoid over~$B$. The construction of $\rm Iso(\Gamma_{k}(\mathscr D))$ may depend on the global frame $\hat\varphi$, a map between two fibers~$\Gamma(\mathscr D)_x$ and~$\Gamma(\mathscr D)_y$ may or not be $\hat\tau_m G$-equivariant for the structure induced by some global frame. However, according to Remark~\ref{rem:2trivializations}, if $\rho(x) = \rho(y)$, then a map between the fibers~$\Gamma(\mathscr D)_x$ and~$\Gamma(\mathscr D)_y$ is $\hat\tau_mG$-equivariant for some induced structure if and only if it is $\hat\tau_mG$-equivariant for any induced structure. Accordingly, the pro-algebraic Lie groupoid ${\rm Iso}_S(\Gamma(\mathscr D))$ does not depend on the choice of the global frame in~$S$.
\end{Remark}

\begin{Example}\label{Example_linear1}
Our main examples are linear differential equations with parameters. Here $P = {\rm GL}_k(\mathbb C)\times \mathbb C^{n+m}$,
$B = \mathbb C^{n+m}$, and $S = \mathbb C^m$. We consider a linear differential system with parameters
\begin{gather}\label{linear_1}
\frac{\partial U}{\partial x_i} = A_iU, \qquad i=1,\ldots,n,
\end{gather}
where $A_i$ are matrices with coefficients in $\mathbb C(x_1,\ldots,x_n,s_1,\ldots,s_m)$. The Frobenius integrability condition imposes here the compatibility equations:
\begin{gather*}\frac{\partial A_i}{\partial x_j} - \frac{\partial A_j}{\partial x_i} = [A_i,A_j].\end{gather*}
Equation \eqref{linear_1} can be seen itself as a sub-bundle of $J_1(P/B)$. In this case,
$J_1(P/B) = \mathbb C^{k^2(n+m)} \times {\rm GL_k}(\mathbb C)\times \mathbb C^{n+m}$, and $U_{x_i}$, $U_{s_j}$, $U$, $x$, $s$ is a system of coordinates. Equations
\begin{gather}\label{linear_or1}
U_{x_j} - A_iU = 0, \qquad i=1,\ldots,n,
\end{gather}
define $\Gamma_1(\mathscr D)$ as a subset of $J_1(P/B)$. A basis of the space of total derivative operators is given by $\partial_{x_i}^{\rm tot}$, $\partial_{s_j}^{\rm tot}$. We differentiate equations \eqref{linear_or1} and obtain
\begin{gather}\label{linear_or2}
U_{x_jx_k} - \frac{ \partial A_i}{\partial x_k}U - A_jA_k U = 0, \\
U_{x_js_k} - \frac{ \partial A_i}{\partial s_k}U - A_j U_{s_k} = 0. \label{linear_or2x}
\end{gather}
Equations \eqref{linear_or1}, \eqref{linear_or2}, and \eqref{linear_or2x} define $\Gamma_2(\mathscr D)$, and so on. Note that $x_i,s_j,U,U_{s_j},U_{s_js_k},\ldots$ are coordinates on $\Gamma(\mathscr D)$.
\end{Example}

\subsubsection[Prolongation of $\mathscr D$ to $\Gamma(\mathscr D)$]{Prolongation of $\boldsymbol{\mathscr D}$ to $\boldsymbol{\Gamma(\mathscr D)}$}
Here we show how to lift the $G$-invariant connection $\mathscr D$ to a projective system of $\tau_{m,k}G$-invariant connections $\mathscr D^{(k)}$ in $\Gamma_k(\mathscr D)$, and thus, a $\hat\tau_m G$-invariant connection in~$\Gamma(\mathscr D)$.

\begin{Lemma}\label{lm:total_equal_lift} Let $X$ be rational vector field tangent to $\mathscr F$. Along $\Gamma_k(\mathscr D)$, the total derivative $X^{\rm tot}$ coincides with $\tilde X^{(k)}$, the jet prolongation of the $\mathscr D$-horizontal lift of~$X$. Therefore $X^{\rm tot}$ and $\tilde X^{(\infty)}$ coincide as derivations of $\mathcal O_{\Gamma(\mathscr D)}$.
\end{Lemma}

\begin{proof}Let $X$ be a rational field tangent to $\mathscr F$. First, we have that the ideal of $\Gamma(\mathscr D)$ is a~differential ideal, so that, $X^{\rm tot}$ defines a derivation of $\mathcal O_{\Gamma(\mathscr D)}$.

Let us fix $j^k_xu \in \Gamma_k(\mathscr D)$. Let us assume that $u$ is a $\mathscr D$-horizontal local section. The vector field $\tilde X$ is tangent to the graph of~$u$. Thus, the flow of $\tilde X$ leaves the section $u$ invariant. Let us consider $\{\sigma_t\}_{t\in \mathbb R}$ the local flow of $X$ around $x$. Then, we have
\begin{gather*}\tilde X^{(\infty)}(j_xu) =\left.\frac{{\rm d}}{{\rm d}t}\right|_{t=0} j_{\sigma_t(x)}u.\end{gather*}
Thus, for any differential function $F$ of arbitrary order, we have
\begin{gather*}\big(X^{\rm tot} F\big)(j_xu) = X_x (F\circ ju) =\left.\frac{{\rm d}}{{\rm d}t}\right|_{t=0} F(j_{\sigma_t(x)}u) = \big(\tilde X^{(\infty)}F\big)(j_xu).\end{gather*}
We have that $X^{\rm tot}$ coincides with $\tilde X^{(\infty)}$ along $\Gamma(\mathscr D)$, and therefore they preserve the order of differential functions.
\end{proof}

\begin{Example}Let us compute ${\partial_{x_j}}^{\rm tot}$ restricted to $\Gamma(\mathscr D)$ in Example \ref{Example_linear1}. Let $u_{ij}$ be the $(i,j)$-entry of coordinate matrix $U$, and $a_{i,jk}$ be the $(j,k)$-entry of matrix $A_i$. The usual expression for the total derivative is
\begin{gather*}\frac{\partial}{\partial x_j}^{\rm tot} =
\frac{\partial}{\partial x_j} + u_{mk;x_j} \frac{\partial}{\partial u_{mk}} + u_{mk;x_\ell x_j}\frac{\partial}{\partial u_{mk;x_\ell}} +
u_{mk;s_\ell x_j}\frac{\partial}{\partial u_{mk;s_\ell}} + \cdots.\end{gather*}
We have that $x_j,s_\ell,u_{mk},u_{mk;s_\ell},\ldots$ is a system of coordinates on $\Gamma(\mathscr D)$. We obtain
\begin{gather*}
\left.\frac{\partial}{\partial x_j}^{\rm tot}\right|_{\Gamma(\mathscr D)} =
\frac{\partial}{\partial x_j} + a_{j,m\ell}u_{\ell k} \frac{\partial}{\partial u_{mk}} + \left(\frac{\partial a_{j,ml}}{\partial s_r} u_{\ell k} + a_{j,m\ell}u_{\ell k;s_r}\right)\frac{\partial}{\partial u_{mk;s_r}} + \cdots.\end{gather*}
Here we see that the total derivative of a function on $\Gamma_k(\mathscr D)$ is also defined on $\Gamma_k(\mathscr D)$.
\end{Example}

\begin{Definition}\label{df:prolD}For each $k = 1,\ldots,\infty$ we call the $k$-th prolongation of $\mathscr D$ to the distribution~$\mathscr D^{(k)}$ spanned by the vector fields of the form $\left.X^{\rm tot}\right|_{\Gamma_k(\mathscr D)}$ where $X$ is tangent to~$\mathscr F$.
\end{Definition}

\begin{Remark}\label{rm:geom_Dk} The above definition ensures $\mathscr D^{(k)}$ is a regular connection on $\Gamma_{k}(\mathscr D)$. Another way of defining $\mathscr D^{(k)}$ in geometrical terms is the following. Let $u$ be a local $\mathscr D$-horizontal section defined around $x\in B$,
\begin{gather*}\mathscr D^{(k)}_{j_x^k u} = {\rm d}\big(j^ku\big)(\mathscr D_x).\end{gather*}
The space $\mathscr D^{(k)}_{j_x^k u}\subset T_{j_x^ku}(\Gamma_k(\mathscr D))$ does not depend on the choice of~$u$. It is confined into the fiber of~$\rho(x)$. Any other $\mathscr D$-horizontal section passing through the same point~$u(x)$ shall coincide with~$u$ along the fiber of~$\rho(x)$.
\end{Remark}

\begin{Proposition}\label{pr:prol_pbundle}
The prolongation $\mathscr D^{(k)}$ is a $\tau_{m,k}G$-invariant connection on $\Gamma_k(\mathscr D)$ with parameters $S$.
The projection $\Gamma_{k+1}(\mathscr D) \to \Gamma_{k}(\mathscr D)$ maps $\mathscr D^{(k+1)}$ to $\mathscr D^{(k)}$.
\end{Proposition}

\begin{proof}
Let $\{X_i\}$ be a commuting basis of the $\mathbb C(B)$-space of rational vector fields tangent to $\mathscr F$. Their values span $\mathscr F$ on the generic point of $B$. In virtue of Lemma \ref{lm:total_equal_lift} we have that $\mathscr D^{(k)}$ is the distribution of vector fields spanned by $\big\{\tilde X_i^{(k)}|_{\Gamma_k(\mathscr D)}\big\}$. From the compatibility of the Lie bracket and the lift of vector fields we obtain that $\mathscr D^{(k)}$ is a Frobenius integrable distribution in $\Gamma(\mathscr D)$. Also, from the properties of the lift we obtain that $\mathscr D^{(k+1)}$ projects onto~$\mathscr D^{(k)}$. The invariance of $\mathscr D^{(k)}$ with respect to $\tau_{m,k}G$ follows automatically from Remark~\ref{rm:geom_Dk}.
\end{proof}

The projective limit $\mathscr D^{(\infty)}= \lim_{\leftarrow} \mathscr D^{(k)}$ is called the \emph{jet prolongation with respect to parameters} of $\mathscr D$. In virtue of Proposition~\ref{pr:prol_pbundle} $\mathscr D^{(\infty)}$ it is a $\hat\tau_m G$ invariant~$\mathscr F$-connection in $\Gamma(\mathscr D)$. Note that the construction of $\mathscr D^{(\infty)}$ is independent of the moving frame $\hat\varphi$ in $S$ that induces the $\hat\tau_m G$-principal bundle structure in $\Gamma(\mathscr D)$. Different moving frames give different actions that are not necessarily equivalent, but at least they are equivalent along the fibers of~$\rho$.

\subsection{The parameterized Galois groupoid}\label{ss:PGal}
 The field of \emph{rational differential invariants} of order $\leq k$ of $\mathscr D$ is the field $\mathbb C(\Gamma_k(\mathscr D))^{\mathscr D^{(k)}}$ of rational first integrals of $\mathscr D^{(k)}$. We have a tower of field extensions
\begin{gather*}\mathbb C(S) \subset \mathbb C(P)^{\mathscr D} \subset \dots \subset \mathbb C(\Gamma_k(\mathscr D))^{\mathscr D^{(k)}} \subset \mathbb C(\Gamma_{k+1}(\mathscr D))^{\mathscr D^{(k+1)}} \subset \cdots,\end{gather*}
and the union of all its members is the field of rational differential invariants of~$\mathscr D$. If~$X$ is a~rational vector field in~$B$ and~$f$ a~rational invariant of $\mathscr D$ of order $\leq k$ then $X^{\rm tot}f$ is a~rational invariant of $\mathscr D$ of order $\leq k+1$. Thus, we have that $\mathbb C(\Gamma(\mathscr D))^{\mathscr D^{(\infty)}}$ is a~differential field with derivations $\mathfrak X(B)^{\rm tot}$. Moreover, by the finiteness of differential subfields (for instance, by \cite[Theorem~14, p.~112]{kolchin1973}) we have that there is some~$\ell$ such that the field of differential invariants is spanned by $\mathbb C(\Gamma_\ell(\mathscr D))^{\mathscr D^{(\ell)}}$ and total differentiation.

Now let us translate, by means of the Galois correspondence from Remark~\ref{Galois_correspondence_3}, this tower of fields into a system of groupoids, and the field of rational differential invariants into a pro-algebraic groupoid.

In order to have a principal structure in $\Gamma(\mathscr D)$ we need to take a moving frame on~$S$. It is well known that there always exists a rational moving frame on $S$. We may, for instance, take the frame induced by a transcendence basis of~$\mathbb C(S)$. Thus, we fix a global frame $\hat\varphi$ in $S$, that we assume to be regular, after restriction to a suitable affine open subset.

\begin{Lemma}The Galois groupoid ${\rm GGal}\big(\mathscr D^{(k)}\big)$ does not depend on the choice of the moving frame $\hat\varphi$ on $S$.
\end{Lemma}

\begin{proof} Let $\hat\varphi$ and $\hat\psi$ two different moving frames on $S$, whose truncations to order $k$ induce respective $\tau_{m,k}G$-structures in $\Gamma_k(\mathscr D)\to B$. Thus, there are two different groupoids of gauge isomorphisms ${\rm Iso}^{\hat\varphi}(\Gamma_k(\mathscr D))$ and ${\rm Iso}^{\hat\psi}(\Gamma_k(\mathscr D))$. However, by Remark~\ref{rm:isoS} we have ${\rm Iso}_{S}(\Gamma_k(\mathscr D))\subset {\rm Iso}^{\hat\varphi}(\Gamma_k(\mathscr D)) \cap {\rm Iso}^{\hat\psi}(\Gamma_k(\mathscr D))$. We also have $\mathbb C(S) \subset \mathbb C(\Gamma_k(\mathscr D))^{\mathscr D^{(k)}}$. By Galois correspondence from Remark~\ref{Galois_correspondence_3} we have that the Galois groupoid of $\mathscr D^{(k)}$ is inside ${\rm Iso}_S(\Gamma_k(\mathscr D))$, hence independent of the choice of the global frame.
\end{proof}

\begin{Definition}\label{def:GaloisgroupP}Let us consider $\mathscr D$ a principal $G$-invariant connection on $P$ with parameters in~$S$.
\begin{enumerate}\itemsep=0pt
\item The Galois groupoid (and group bundle) of order $k$ with parameters of $\mathscr D$
is the Galois groupoid (and group bundle) of gauge isomorphisms (automorphisms) of $\Gamma_k(\mathscr D)$ that fixes rational differential invariants of $\mathscr D$ of order $\leq k$,
\begin{gather*}{\rm PGal}_k(\mathscr D) = {\rm GGal}\big(\mathscr D^{(k)}\big), \qquad {\rm Gal}_k(\mathscr D) = {\rm Gal}\big(\mathscr D^{(k)}\big).\end{gather*}
\item The Galois groupoid with parameters of $\mathscr D$ is the groupoid of gauge isomorphisms
of $\Gamma(\mathscr D)$ that fixes rational differential invariants of $\mathscr D$,
\begin{gather*}{\rm PGal}_{\infty}(\mathscr D) = \lim_{\infty \leftarrow k}{\rm GGal}\big(\mathscr D^{(k)}\big).\end{gather*}
\item The Galois group with parameters of $\mathscr D$ is the diagonal part of the Galois groupoid with parameters of $\mathscr D$,
\begin{gather*}{\rm Gal}_{\infty}(\mathscr D) = {\rm PGal}_{\infty}(\mathscr D)^{\rm diag}.\end{gather*}
\end{enumerate}
\end{Definition}

\subsection[Differential structure of the ${\rm Gau}(\Gamma(\mathscr D))$]{Differential structure of the $\boldsymbol{{\rm Gau}(\Gamma(\mathscr D))}$}\label{ss:PG}
The Galois group with parameters, as in Definition~\ref{def:GaloisgroupP}, is rational subgroup bundle of \linebreak ${\rm Gau}(\Gamma(\mathscr D))$. In what follows we present its natural differential group structure.

The fibers of $J(P/B)\to B$ are principal homogeneous spaces for the fibers of the group bundle $J(G/B)\to B$ (see Section~\ref{sss:hom_fibers}). We can consider the bundle ${\rm Gau}(J(P/B))\to B$ of gauge automorphisms as the quotient $(J(P/B)\times_B J(P/B))/J(G/B)$. This quotient exists as a pro-algebraic variety (at least on the generic point of~$B$): it can be constructed, for instance, trivializing the group bundle $J(G/B)\to B$ by means of a moving frame in~$B$.

\begin{Lemma}\label{lm:JG}Let $E \to B$ be a fiber bundle and $G$ an algebraic group acting on~$E$ and preserving the projection to~$B$. Assume $E/G$ exists as smooth algebraic variety. Then there is a canonical isomorphism $J(E/B) / J(G/B) \simeq J((E/G)/B)$ of bundles over~$B$.
\end{Lemma}

\begin{proof}Let $f$ be the projection $E \to E/G$ and $jf$ the induced map $J(E/B) \to J((E/G)/B)$. Because $jf(j\sigma) = j(f(\sigma))$, this maps induces the wanted isomorphism.
\end{proof}

\begin{Proposition}\label{pr:jcommutes}For the principal bundle $\pi\colon P \to B$ there is a canonical isomorphism ${\rm Gau }(J(P/B)) \simeq J({\rm Gau}(P)/B)$ of bundles over~$B$.
\end{Proposition}
\begin{proof}It suffices to consider the case $E = P\times_B P$ in Lemma~\ref{lm:JG}.
\end{proof}

From this, we know that ${\rm Gau}(J(P/B))$ has a structure of differential algebraic group. Let us consider now the bundle of jets of $\mathscr D$-horizontal sections $\Gamma(\mathscr D)\subset J(P/B)$. It is a \emph{differential subvariety} of~$J(P/B)$ in the sense that it is the zero locus of the differential ideal $\mathcal E(\mathscr D)\subset \mathcal O_{J(P/B)}$ (see Section~\ref{ss:prolpc}).

\begin{Theorem}\label{th:Gau_is_adj}The group bundle of gauge automorphisms $\Gamma(\mathscr D)$ is identified with the differential algebraic subgroup of gauge symmetries of~$\mathscr D$, defined by the group connection ${\rm Adj}(\mathscr D)$,
\begin{gather*}{\rm Gau}(\Gamma(\mathscr D))\simeq \Gamma({\rm Adj}(\mathscr D))\subset J({\rm Gau}(P)/B).\end{gather*}
\end{Theorem}

\begin{proof}The differential subvariety $\Gamma(\mathscr D) \subset J(P/B)$ is principal for the action of the subgroup bundle $J(G/S)\times B \subset J(G/B)$ (note that two $\mathscr D$-horizontal sections are related by a right translation which is constant along the fibers of $\rho$), this being a differential algebraic subgroup
of~$J(G/B)$.

By means of the general construction of ${\rm Gau}(P) = (P\times_B P)/G$ and Proposition~\ref{pr:jcommutes} we have a commutative diagram
\begin{gather*}\xymatrix{
\Gamma(\mathscr D)\times_B \Gamma(\mathscr D) \ar[r]\ar[d] &
J(P/B)\times J(P/B) \ar[d] \\
{\rm Gau}(\Gamma(\mathscr D)) \ar[r]^-{i} &
J({\rm Gau}(P)/B).
}\end{gather*}
It is clear that if a pair $(j_xu,j_xv)$ of jets of $\mathscr D$-horizontal sections are related by an element of~$J(G/B)_x$, then this element is constant along the fibers of $\rho$. Therefore, the map $i$ is injective. Finally, an element $j_x\sigma \in J(\rm Gau(P)/B)$ is in the image of $i$ if and only if it transform $\mathscr D$-horizontal sections into $\mathscr D$-horizontal sections, if and only if it is a jet of a gauge symmetry of~$\mathscr D$, if and only if it is ${\rm Adj}(\mathscr D)$-horizontal.
\end{proof}

\begin{Remark}\label{rm:compatibility}
The action of ${\rm Gau}(\Gamma(\mathscr D))$ on $\Gamma(\mathscr D)$ is, by definition, compatible with the differential structure, in the sense that the co-action
\begin{gather*}\alpha^*\colon \ \mathcal O_{\Gamma(\mathscr D)} \to
\mathcal O_{{\rm Gau}(\Gamma(\mathscr D))} \otimes_{\mathbb C(B)}
\mathcal O_{\Gamma(\mathscr D)}\end{gather*}
is compatible with the total derivative operators.
\end{Remark}

\begin{Theorem}\label{th:Gal_is_DAG} The parameterized Galois group bundle ${\rm Gal}_{\infty}(\mathscr D)$ is a differential algebraic subgroup of $J({\rm Gau}(P)/B)$.
\end{Theorem}

\begin{proof}Let us compute the equations of ${\rm Gal}_{\infty}(\mathscr D)$ inside ${\rm Gau}(\Gamma(\mathscr D))$ in order to verify that it is the zero locus of a~differential Hopf ideal in $\mathcal O_{{\rm Gau }(\Gamma(\mathscr D))}$. From the group bundle structure we already know that the ideal of ${\rm Gal}_{\infty}(\mathscr D)$ is a Hopf ideal; it suffices to check that it is a differential ideal.

By definition,
${\rm Gal}_{\infty}(\mathscr D)$ is the group of gauge automorphisms commuting with rational differential invariants
$\mathbb C(\Gamma(\mathscr D))^{\mathscr D^{(\infty)}}$. This field is differentially finitely generated over $\mathbb C(S)$, let us take a system of generators.
\begin{gather*}\mathbb C(\Gamma(\mathscr D))^{\mathscr D^{(\infty)}} =
\mathbb C(S)\langle h_1,\ldots,h_k \rangle.\end{gather*}
There is some $N$ big enough such that the differential functions $h_i$ belong to $\mathbb C(\Gamma_N(\mathscr D))$. By construction (and restricting our considerations to a suitable open subset of $B$) $\Gamma_N(\mathscr D)$ is an affine algebraic variety, so that we have
\begin{gather*}h_i = \frac{P_i}{Q_i}\end{gather*}
with $P_i$, $Q_i$ in $\mathcal O_{\Gamma(\mathscr D)}$. Now, we define
\begin{gather*}F_i = \alpha^*(P_i)(1\otimes Q_i) - \alpha^*(Q_i)(1\otimes P_i),\end{gather*}
which are elements of $\mathcal O_{{\rm Gau}(\Gamma(\mathscr D))} \otimes_{\mathbb C(B)} \mathcal O_{\Gamma(\mathscr D)}$. We consider $\mathcal I$ the radical differential ideal $\{F_1,\ldots,F_k\}$. From Remark \ref{rm:compatibility} we have that the zero locus of $\mathcal I$ is ${\rm Gal}_{\infty}(\mathscr D)\times_B \Gamma(\mathscr D)$. Therefore, the ideal of the Galois group bundle is the differential-Hopf ideal $i^*(\mathcal I)$ for the canonical inclusion $i^*\colon \mathcal O_{{\rm Gau}(\Gamma(\mathscr D))}\to \mathcal O_{{\rm Gau}(\Gamma(\mathscr D))} \otimes_{\mathbb C(B)} \mathcal O_{\Gamma(\mathscr D)}$.
\end{proof}

\subsection{Isomonodromic deformations}\label{ss:isom}

An \emph{isomonodromic deformation} of the connection $\mathscr D$ with parameters in $S$ is a partial $G$-invariant connection $\tilde{\mathscr D}$ extending $\mathscr D$.
If we take $\tilde{\mathscr F} = \pi_*\big(\tilde{\mathscr D}\big)$, then $\tilde{\mathscr D}$ is a partial $\tilde{\mathscr F}$-connection. Paths along the leaves of $\rho_*(\tilde{\mathscr F})$ are called \emph{isomonodromic paths}. The main application is that the monodromy representation of the connection with parameters is constant (up to conjugation) along isomonodromic paths.

\begin{Remark}In the analytic context, the restriction of $\mathscr D$ to a leaf of $\tilde{\mathscr F}$ is an isomonodromic family. Under some assumptions we have that $\mathscr D$ is an isomonodromic family if and only if it admits an isomonodromic deformation which is a $TB$-connection (see \cite[Section~5]{cassidy2005galois} and references therein).
\end{Remark}

An important point about isomonodromic deformations is that they give us first order equations for the Galois group with parameters. The following lemma is a generalization (in our geometrical setting) of Proposition~5.4 in \cite{cassidy2005galois}.

\begin{Lemma}\label{lm:key} If $\tilde{\mathscr D}$ is an isomonodromic deformation of~$\mathscr D$, then ${\rm Gal}_{\infty}(\mathscr D) \subset {\Gamma}\big({\rm Adj}\big(\tilde{\mathscr D}\big)\big)$.
\end{Lemma}

\begin{proof}It suffices to verify that ${\rm Gal}_k(\mathscr D)\subset {\Gamma}_k\big({\rm Adj}\big(\tilde{\mathscr D}\big)\big)$ for any finite~$k$. Any $\tilde{\mathscr D}$-horizonal section is, in particular, $\mathscr D$-horizontal. It follows that $\Gamma_k\big(\tilde{\mathscr D}\big) \subset \Gamma_k(\mathscr D)$.

Let us consider a moving frame in $S$, so that we may make use of the $\tau_{m,k}G$-structure in~$\Gamma(\mathscr D)$. From the construction of the prolongation $\mathscr D^{(k)}$ we have that for any $j_x^ku\in \Gamma_k\big(\tilde{\mathscr D}\big)$ there is an inclusion of vector spaces $\mathscr D^{(k)}_{j_x^ku}\subset \tilde{\mathscr D}_{j_x^ku}^{(k)}$. Therefore $\Gamma_k\big(\tilde{\mathscr D}\big)$ is tangent to the distribution $\mathscr D^{(k)}$. It is in fact a reduction of structure group to $\tau_{m-d,k}G$, where~$d$ is the codimension of $\mathscr D$ in~$\tilde{\mathscr D}$.

From there, we have that ${\rm Iso}(\Gamma_k\big(\tilde{\mathscr D}\big))$ is a subgroupoid of ${\rm Iso}(\Gamma_k(\mathscr D))$ and it is tangent to the distribution $\mathscr D^{(k)} \boxtimes \mathscr D^{(k)}$, as defined in Remark~\ref{rm:DxD}. From the topological construction of the Galois groupoid (see Remark~\ref{rm:topGalois}) we have ${\rm GGal}\big(\mathscr D^{(k)}\big) \subset {\rm Iso}\big(\Gamma_k\big(\tilde{\mathscr D}\big)\big)$. Now, by taking the diagonal parts
we have ${\rm Gal}_k(\mathscr D)\subset \Gamma_k\big({\rm Adj}\big(\tilde{\mathscr D}\big)\big)$.
\end{proof}

Theorems \ref{th:KC} and \ref{th:Mal} describe all \emph{Zariski dense} differential subgroups $H$ of $J({\rm Gau}(P)/B)$ for simple algebraic group $G$. There exists a unique partial $G$-invariant connection $\tilde{\mathscr D}$ such that $H = \Gamma\big({\rm Adj}\big(\tilde{\mathscr D}\big)\big)$. Note that
${\rm Gal}_{\infty}(\mathscr D)$ is Zariski dense in $J({\rm Gau}(P)/B)$ if and only if ${\rm Gal}(\mathscr D) = {\rm Gau}(P)$. We can state our main result.

\begin{Theorem}\label{th:main}
Let $P\to B$ be an affine principal bundle with simple group $G$, $\rho\colon B\to S$ a~dominant map, $\mathscr F = \ker({\rm d}\rho)$, and $\mathscr D$ a~principal $G$-invariant connection with parameters in $S$. Let us assume that the Galois group of $\mathscr D$ is~${\rm Gau}(P)$. Then, there is a biggest isomonodromic deformation $\tilde{\mathscr D}$ of $\mathscr D$ and the Galois group with parameters ${\rm Gal}_{\infty}(\mathscr D)$ is the group $\Gamma\big({\rm Adj}\big(\tilde{\mathscr D}\big)\big)$ of gauge symmetries of $\tilde{\mathscr D}$.
\end{Theorem}

\begin{proof}From Lemma \ref{lm:key} and Theorem~\ref{th:KC} it is clear that the Galois group with parameters is the group of gauge symmetries of a certain isomonodromic deformation $\tilde{\mathscr D}$ of $\mathscr D$. It suffices to check that any other isomonodromic deformation is extended by $\tilde{\mathscr D}$. Let $\mathscr D_1$ be an isomonodromic deformation of $\mathscr D$. By Lemma~\ref{lm:key} we have ${\rm Gal}_\infty(\mathscr D) = \Gamma\big(\rm Adj\big(\tilde{\mathscr D}\big)\big) \subset {\Gamma}({\rm Adj}(\mathscr D_1))$. From there, and Theorem~\ref{th:Mal} we have $\mathscr D_1 \subset \tilde{\mathscr D}$.
\end{proof}

\begin{Example}\label{ex:fuchsian}
Let us consider $B = \mathbb C^{k+1} \times \mathfrak{sl}_n(\mathbb C)^k$ with coordinates $x,a_1,\ldots,a_k,A_1,\ldots,A_k$; $S= \mathbb C^k\times \mathfrak{sl}_n(\mathbb C)^k$ with the obvious projection (forgetting the coordinate $x$), and $P = {\rm SL}_n(\mathbb C) \times B$. The general Fuchsian system with $k$ singularities and trace free matrices
\begin{gather}\label{eq:Fuchsian}
\frac{\partial U}{\partial x} = \left(\sum_{i=1}^n \frac{A_i}{x-a_i}\right)U
\end{gather}
is a ${\rm SL}_n(\mathbb C)$-invariant connection $\mathscr D$ on $P$ with parameters in $S$. The general isomonodromic deformation (see, for instance, \cite[Section~3.5]{iwasaki2013gauss}) admitted by~\eqref{eq:Fuchsian} is the well-known Schlessinger system. It consists in the system of equations formed by~\eqref{eq:Fuchsian} together with
\begin{gather}
\frac{\partial A_i}{\partial a_j} = \frac{[A_i,A_j]}{a_i-a_j}, \label{26}\\
\frac{\partial A_i}{\partial a_i} = -\sum_{i\neq j}\frac{[A_i,A_j]}{a_i-a_j}, \label{27} \\
\frac{\partial U}{\partial a_i} = \frac{A_i}{x-a_i}U. \label{28}
\end{gather}
Equations \eqref{26} and \eqref{27} define the foliation $\rho_*(\tilde{\mathscr F})$ of isomonodromic paths in $S$, and equa\-tions~\eqref{eq:Fuchsian}--\eqref{28} define the isomonodromic deformation $\tilde{\mathscr D}$ as a partial $\tilde{\mathscr F}$-connection. Thus, the Galois group with parameters of~\ref{eq:Fuchsian} is the differential algebraic subgroup of $J({\rm SL}_n(\mathbb C)/B)$ given by equations~\eqref{26},~\eqref{27}, and
\begin{gather*}
\frac{\partial \sigma}{\partial x} = \left[\sum_{i=1}^n \frac{A_i}{x-a_i},\sigma \right], \qquad
\frac{\partial \sigma}{\partial a_i} = \left[\frac{A_i}{x-a_i}, \sigma \right].
\end{gather*}
\end{Example}

\section{Gauss hypergeometric equation}\label{s:GHyp}

Let us consider Gauss hypergeometric equation
\begin{gather}\label{HG}
\quad x(1-x)\frac{{\rm d}^2u}{{\rm d}x^2} + \{\gamma -(\alpha+\beta+1) x \} \frac{{\rm d}u}{{\rm d}x} - \alpha\beta u=0.
\end{gather}

We see equation \eqref{HG} as a linear connection with parameters in the following way. We set $P = {\rm GL}_2(\mathbb C) \times B$, $B = \mathbb C_{x}\times S$ and $S = \mathbb C^{3}_{\alpha,\beta,\gamma}$ with the obvious projections. Then, the hypergeometric equation is the ${\rm GL}_2(\mathbb C)$-invariant partial connection $\mathcal H$ in $P$ defined by the differential system
\begin{gather}
{\rm d}\left(\begin{matrix}
u_{11} & u_{12} \\
u_{21} & u_{22}
\end{matrix} \right) -
\left(\begin{matrix}
0 & 1 \\
\dfrac{\alpha\beta}{x(1-x)} & \dfrac{(\alpha+\beta+1)x - \gamma}{x(1-x)}
\end{matrix} \right)
\left(\begin{matrix}
u_{11} & u_{12} \\
u_{21} & u_{22}
\end{matrix} \right) = 0, \nonumber\\
{\rm d}\alpha = 0, \qquad {\rm d}\beta = 0, \qquad {\rm d}\gamma = 0.\label{HGmatrix}
\end{gather}

\subsection{Projective reduction}
Let us set $P_{\rm proj} = {\rm PGL}_2(\mathbb C)\times B$. There is a canonical quotient map $P\to P_{\rm proj}$ which is a~morphism of principal bundles. The projection of $\mathcal H$ is a ${\rm PGL}_2(\mathbb C)$-invariant partial connection~$\mathcal H_{\rm proj}$ on~$P_{\rm proj}$. This connection has a nice geometric interpretation. Let us consider $J_2^*(\mathbb C^1,\mathbb P_1)$ the space of $2$-jets of local diffeomorphisms from $\mathbb C^1$ to $\mathbb P_1$. Let us choose an affine embedding~$\varphi$ of~$\mathbb C_1$ in $\mathbb P_1$, for instance $\varphi(\varepsilon) = [\varepsilon : 1]$. Let us consider the action of~${\rm PGL}_2(\mathbb C)$ in
$\mathbb P_1$ on the right side, by setting
\begin{gather*}[t_0:t_1]\star \left[\begin{matrix}
a & b \\ c & d
\end{matrix} \right] = [at_0 + ct_1: bt_0 + dt_1].\end{gather*}
We get an isomoporphism of ${\rm PGL}_2(\mathbb C)$ bundles
\begin{gather*}P_{\rm proj} \to J^*_2(\mathbb C, \mathbb P_1) \times \mathbb C^3_{\alpha,\beta,\gamma}, \qquad (\sigma,x,\alpha,\beta,\gamma) \mapsto \big(j^2_x (\varphi\star \sigma),\alpha,\beta,\gamma\big).\end{gather*}
Through such isomorphism, the reduced connection is seen as a third order differential equation in the projective space, the well known Schwartzian equation: if $\{u_1,u_2\}$ is a local basis of solutions of~\eqref{HG} then $y = \frac{u_1}{u_2}$ satisfies
\begin{gather*}
\quad {\rm Schw}_x(y) = \nu(\alpha,\beta,\gamma ; x),
\end{gather*}
with
\begin{gather*} \nu(\alpha,\beta,\gamma ;x) = \frac{\gamma(2-\gamma)}{4x^2} +\frac{1-(\gamma - \alpha - \beta)^2}{4(1-x^2)}+\frac{\gamma(1-\gamma+\alpha+\beta)-2\alpha\beta}{2 x(1-x)}.\end{gather*}

\begin{Lemma}\label{no_iso}For very general values of $(\alpha,\beta,\gamma)\in S$ there is no smooth isomonodromic path of $\mathcal H_{\rm proj}$ passing through $(\alpha,\beta,\gamma)$.
\end{Lemma}

\begin{proof}Here we apply the description of the monodromy of~\eqref{HG} in terms of the parameters, as explained in~\cite{iwasaki2013gauss}. The eigenvalues of the monodromy matrices at $0$, $1$, $\infty$ are respectively $\big\{1,{\rm e}^{-2\gamma \pi{\rm i}}\big\}$, $\big\{1, {\rm e}^{2(\gamma-\alpha-\beta)\pi{\rm i}}\big\}$, $\big\{{\rm e}^{2\alpha\pi{\rm i}},{\rm e}^{2\beta \pi{\rm i}}\big\}$. The transcendental functions
\begin{gather*}f = {\rm e}^{2\gamma \pi{\rm i}} + {\rm e}^{-2\gamma\pi{\rm i}}, \qquad g = {\rm e}^{2(\gamma-\alpha-\beta)\pi{\rm i}}+ {\rm e}^{2(\alpha+\beta-\gamma)\pi{\rm i}}, \qquad
h = {\rm e}^{2(\alpha-\beta)\pi{\rm i}} + {\rm e}^{2(\beta-\alpha)\pi{\rm i}}\end{gather*}
are, by construction, invariants of the monodromy of $\mathcal H_{\rm proj}$. Thus, they are constant along any isomonodromic path. A simple calculus shows that they are functionally independent on the very general point of $S$. It yields the result.
\end{proof}

\begin{Corollary} The differential Galois group with parameters of $\mathcal H_{\rm proj}$ consists in all the gauge symmetries of $\mathcal H_{\rm proj}$, ${\rm Gal}_{\infty}(\mathcal H_{\rm proj}) = {\rm Gau}(\Gamma(\mathcal H_{\rm proj}))$.
\end{Corollary}

\begin{proof}It is a direct consequence of Lemma~\ref{no_iso} and Theorem~\ref{th:main}.
\end{proof}

\subsection{Differential equation for the determinant}

Another reduction of \eqref{HG} is given by the quotient map $P\to P_{\det} = \mathbb C^*\times B$, $(U,x,\alpha,\beta,\gamma)\mapsto(w = {\det}(U),x,\alpha,\beta,\gamma)$. The projection is a morphism of principal bundles and it sends $\mathcal H$ to a~partial connection $\mathcal H_{\det}$ in $P_{\det}$ defined by the differential system
\begin{gather*}
{\rm d}w -\frac{(\alpha+\beta+1)x - \gamma}{x(1-x)}w\,{\rm d}x = 0,\qquad {\rm d}\alpha = 0, \qquad {\rm d}\beta = 0, \qquad {\rm d}\gamma = 0.
\end{gather*}
It is convenient to consider an affine change of variables in the parameter space, by setting
\begin{gather*}
a = -\gamma,\qquad b = \alpha + \beta - \gamma +1, \qquad c = \alpha.
\end{gather*}
The equation for horizontal sections of $\mathcal H_{\det}$ is now written as
\begin{gather*}
\frac{{\rm d}w}{{\rm d}x} = \left(\frac{a}{x} - \frac{b}{1-x}\right)w.
\end{gather*}
And its general solution is
\begin{gather*}w = f(a,b,c)x^a(1-x)^b.\end{gather*}
Let us compute the Galois group with parameters of such equation. Note that
\begin{gather*}\mathbb C(\Gamma(\mathcal H_{\det})) = \mathbb C(x,a,b,c,w,w_a,w_b,w_c,w_{aa},\ldots) = \mathbb C(x,a,b,c)\langle w \rangle.\end{gather*}
Parameter $c$ does not apear on the equation. This means that
in the $\mathbb C(\Gamma(\mathcal H_{\det}))$ the differential fuction $w_c$ is an invariant. Thus,
\begin{gather*}\mathbb C\langle w_c \rangle \subset \mathbb C(\Gamma({\mathcal H}_{\det}))^{{\mathcal H}_{\det}^{\infty}}.\end{gather*}
For the rest of the computation let us ignore the derivatives w.r.t.\ the parameter $c$ and let us consider the equation as dependent only of parameters $a$ and $b$. This is equivalent to projecting the equation onto a quotient basis~$B'$. Let us consider the equation for the second prolongation~${\mathcal H}_{\det}^{(2)}$. It is
\begin{gather*}
w_x = \left(\frac{a}{x} - \frac{b}{1-x}\right)w, \\
w_{ax} = \frac{1}{x}w + \left(\frac{a}{x} - \frac{b}{1-x}\right)w_a, \\
w_{bx} = \frac{1}{1-x}w + \left(\frac{a}{x} - \frac{b}{1-x}\right)w_b, \\
w_{aax} = \frac{2}{x}w_a + \left(\frac{a}{x} - \frac{b}{1-x}\right)w_{aa},\\
w_{abx} = \frac{1}{1-x}w_a + \frac{1}{x}w_b + \left(\frac{a}{x} - \frac{b}{1-x}\right)w_{ab}, \\
w_{bbx} = \frac{2}{1-x}w_b + \left(\frac{a}{x} - \frac{b}{1-x}\right)w_{bb}.
\end{gather*}

At this level, they appear three rational first integrals of ${\mathcal H}_{\det}^{(2)}$, namely,
\begin{gather*}R = \frac{w_{aa}}{w} - \frac{w_a^2}{w^2}, \qquad S = \frac{w_{ab}}{w} - \frac{w_aw_b}{w^2}, \qquad T = \frac{w_{bb}}{w} - \frac{w_b^2}{w^2}.\end{gather*}

Thus,
$\mathbb C(\Gamma({\mathcal H}_{\det}))^{{\mathcal H}_{\det}^{\infty}}$ contains the $\mathbb C^*$-invariant~$\mathfrak X(B)^{\rm tot}$-differential field spanned by~$w_c$,~$R$,~$S$ and~$T$. Let us compute the differential $B$-group associated to this field.

The gauge group $J({\rm Gau}(P/B))$ is in this case isomorphic $J(B,\mathbb C^*)$. On the other hand, we have that the adjoint equation in dimension $1$ vanish, so we also have $\sigma_x = 0$. From the rational invariants we have that the equations for the Galois group ${\rm Gal}_{\infty}(\mathcal H_{\det})$ should include
\begin{gather*}
\sigma_x = 0, \qquad \sigma_c = 0, \qquad \sigma_{aa} = \frac{\sigma_a^2}{\sigma}, \qquad \sigma_{ab} = \frac{\sigma_a\sigma_b}{\sigma}, \qquad \sigma_{bb} = \frac{\sigma_b^2}{\sigma}.
\end{gather*}
The three last equations are equivalent to say that $\sigma$ has constant logaritmic derivative with respect to $a$ and $b$. So, these equations define a~differential algebraic group, which is in fact a~group bundle with fibers of dimension~$3$ over~$B$,
\begin{gather*}\mathcal G = \{j_p\sigma \colon \sigma_x = \sigma_c = \partial_a(\sigma_a/\sigma) = \partial_b(\sigma_a/\sigma) = \partial_b(\sigma_b/\sigma) = 0\}.\end{gather*}
Let us observe that
\begin{gather*}\mathcal G \simeq \cdots \simeq \mathcal G_{k} \simeq \cdots \simeq \mathcal G_3 \simeq \mathcal G_{2} \simeq \mathcal G_1\to B.\end{gather*}

Let us see that the Galois group is exactly this group.

\begin{Lemma}\label{Lemma_3}Let us fix $a$ and $b$ complex numbers $\mathbb Q$-linearly independent. Let us consider the differential system with coefficients in $\mathbb C(x)$:
\begin{gather*}
\frac{{\rm d}}{{\rm d}x}\left(
\begin{matrix}
w_0 \\ w_1 \\ w_2
\end{matrix}\right) =
\left(
\begin{matrix}
\dfrac{a}{x}-\dfrac{b}{1-x} & 0 &0 \vspace{1mm}\\
\dfrac{1}{x} & \dfrac{a}{x}-\dfrac{b}{1-x} & 0\vspace{1mm}\\
\dfrac{-1}{1-x} & 0 & \dfrac{a}{x}-\dfrac{b}{1-x}
\end{matrix}\right)
\left(
\begin{matrix}
w_0 \\ w_1 \\ w_2
\end{matrix}\right).
\end{gather*}
Its Galois group is a connected group of dimension $3$.
\end{Lemma}

\begin{proof}A fundamental matrix of solutions
\begin{gather*}\left(
\begin{matrix}
x^a(1-x)^b & 0 & 0 \\
\log(x)x^a(1-x)^b & x^a(1-x)^b & 0 \\
\log(1-x)x^a(1-x)^b & 0 & x^a(1-x)^b
\end{matrix}
\right).
\end{gather*}
We may compute its monodromy group. For rationally independent $a$ and $b$ its Zariski closure is connected group of dimension~$3$.
\end{proof}

We have that ${\rm Gal}_1\big(\mathcal H_{\det}^{(1)}\big)_{(x,a,b,c)}$ must contain the monodromy at $(x,a,b,c)$ of the connection~$P_{\det}$.
By Lemma \ref{Lemma_3} we have that the Zariski closure of the monodromy on $(x,a,b,c)$ contains~$\mathcal G_{1,(x,a,b,c)}$. Therefore we have
\begin{gather*}{\rm Gal}_1\big(\mathcal H_{\det}^{(1)}\big)_{(x,a,b,c)} \subset \mathcal G_{1,(x,a,b,c)} \subset {\rm Gal}_1\big(\mathcal H_{\det}^{(1)}\big)_{(x,a,b,c)}.\end{gather*}
Thus, we have the equality. The equations of the Galois group can be easily integrated, so be obtain a description of the functions whose jets are elements of
${\rm Gal}_{\infty}(\mathcal H_{\det})$.

\begin{Proposition}\label{pr:abc} The differential Galois group with parameters of $\mathcal H_{\det}$ is
\begin{gather*}{\rm Gal}_{\infty}(\mathcal H_{\det}) = \big\{j_{(x_0,\alpha_0,\beta_0,\gamma_0)}\sigma \colon
\sigma(x,\alpha,\beta,\gamma) = \exp (\mu_0 + \mu_1 \gamma + \mu_2 (\alpha + \beta) ),
\, \mu_0,\mu_1,\mu_2\in\mathbb C\big\},\end{gather*}
which is the set of differential zeroes of the Hopf differential ideal
\begin{gather*}\left\{\partial_x\sigma, \partial_c \sigma, \partial_a\frac{\partial_a\sigma}{\sigma},
\partial_a\frac{\partial_b \sigma}{\sigma}, \partial_b \frac{\partial_b \sigma}{\sigma} \right\} \subset \mathcal O_{J(\mathbb C^*/B)},\end{gather*}
where
\begin{gather*}
\partial_a = -\partial_\beta-\partial_\gamma, \qquad \partial_b = \partial_\beta, \qquad \partial_c = \partial_\alpha - \partial_\beta.
\end{gather*}
\end{Proposition}

\subsection{Computation of the Galois group}

We have a $2$-covering group morphism
\begin{gather*}\varphi \colon \ {\rm GL}_2(\mathbb C) \to {\rm PGL}_2(\mathbb C)\times \mathbb C^*,
\qquad A \mapsto ([A], \det(A)).\end{gather*}
Its kernel of is $\left\{-1,1\right\}$. Let us set $G_{\rm red} = {\rm GL}_2(\mathbb C)/\{-1,1\}$, $P_{\rm red} = P/\{-1,1\}$. So $P_{\rm red}$ is $G_{\rm red}$-principal bundle. The hypergeometric distribution $\mathcal H$ induces $\mathcal H_{\rm red}$ in $P_{\rm red}$.

We have an isomorphism
\begin{gather*}\psi\colon \ P_{\rm red} \xrightarrow{\sim} P_{\rm proj} \times_B P_{\det},\end{gather*}
that yield us a decoupling $\psi_{*}({\mathcal H}_{\rm red}) = \mathcal H_{\rm proj} \times_B \mathcal H_{\det}$. From the commutation of direct products and jets, we have that ${\rm Gal}_\infty(\mathcal H_{\rm red})$ is a differential algebraic subgroup of ${\rm Gal}_{\infty}(\mathcal H_{\rm proj})\times_B {\rm Gal}_\infty(\mathcal H_{\det})$ that projects surjectively onto both components. We will make use of the following.

\begin{Lemma}\label{lm:simple_dsimple}Let $G$ be a simple algebraic group and $B$ an algebraic variety. $J(G/B)$~is simple as differential algebraic group.
\end{Lemma}

\begin{proof}Let $H\lhd J(G/B)$ be a differential algebraic subgroup. Then $H_0$ is rational group subbundle of $G\times B\to B$, therefore either $H_0 = G\times B$ or $H_0 = \{{\rm Id}\}\times B$. In the latter case $H$ is the identity subgroup bundle $J(G/B)$. In the former case, by Theorem \ref{th:KC} we have $H = \Gamma(\mathscr C)$ for a singular foliation $\mathscr F$ on $B$ and a group $\mathscr F$-connection $\mathscr C$.

Let us recall the classical decomposition of the tangent group $TG \simeq {\rm Lie}(G) \ltimes G$. Let us fix~$x$ a~regular point of $\mathscr F$. The $1$-jet at $x$ of a~function from $B$ to $G$ can be identified with its derivative which is a linear map from $T_xB$ to a fiber of $TG\to G$. Therefore we have semidirect product decomposition
\begin{gather*}J_1(G/B)_x\simeq {\rm Lin}_{\mathbb C}(T_xB,{\rm Lie}(G))\ltimes G.\end{gather*}
On the other hand, let us consider $\mathscr L$ the leaf of $\mathscr F$ that passes through $x$. The same argument applies, so we have
\begin{gather*}J_1(G/{\mathscr L})_x\simeq {\rm Lin}_{\mathbb C}(\mathscr F_x,{\rm Lie}(G))\ltimes G.\end{gather*}
The value of the group $\mathscr F$-connection $\mathscr C$ at $x$ can be seen as a section $\mathscr C_x$ from $G$ to $J_1(G/{\mathscr L})_x$ that assigns to any initial condition $g$ in $G$ the only $\mathscr C$-horizontal leaf that passes through $(x,g)$. In the following diagram, where $\pi_x$ is the restriction to the leaf $\mathscr L$, we have $H_{1,x} = \pi_x^{-1}(\mathscr C_x(G))$,
\begin{gather*}\xymatrix{
0 \ar[r] &
{\rm Lin}_{\mathbb C}(T_xB,{\rm Lie}(G)) \ar[r] \ar[d] &
J_1(G/B)_x \ar[r] \ar[d]^-{\pi_x} &
G \ar[r] \ar[d] &
\{{\rm Id}\}\\
0 \ar[r] &
{\rm Lin}_{\mathbb C}(\mathscr F_x,{\rm Lie}(G)) \ar[r] &
J_1(G/\mathscr L)_x \ar[r] &
G \ar[r] \ar@/^1.0pc/[l]^-{\mathscr C_x}
& \{{\rm Id}\}.
}\end{gather*}
If $H_{1,x}\lhd J_1(G/B)_x$ then $\mathscr C_x(G)\lhd J_1(G/\mathscr L)_x$. However, this would produce an isomorphic direct product decomposition $J_1(G/\mathscr L)_x \simeq {\rm Lin}(\mathscr F_x, {\rm Lie}(G)) \times G$, which does not exist as the adjoint representation is not trivial. Hence, $H_{1,x}$ is not a normal subgroup, except in the case in which~$\mathscr F$ is the foliation by points.
\end{proof}

\begin{Lemma}\label{lem}${\rm Gal}_\infty(\mathcal H_{\rm red}) = {\rm Gal}_{\infty}(\mathcal H_{\rm proj})\times_B {\rm Gal}_\infty(\mathcal H_{\det})$.
\end{Lemma}

\begin{proof}Here we follow a reasoning of Kolchin \cite{kolchin1968}. Consider the projections $\pi_{\rm proj}\colon {\rm Gal}_\infty(\mathcal H_{\rm red})
\allowbreak \to {\rm Gal}_{\infty}(\mathcal H_{\rm proj})$ and $\pi_{\rm proj}\colon {\rm Gal}_\infty(\mathcal H_{\rm red}) \to {\rm Gal}_{\infty}(\mathcal H_{\det})$. The kernel of $\pi_{\rm red}$ is a differential algebraic group of the form
\begin{gather*}\ker(\pi_{\rm proj}) = {\rm Id} \times_B K^{\det},\end{gather*}
where $K^{\det}$ is a differential algebraic subgroup of ${\rm Gal}_\infty(\mathcal H_{\det})$. Moreover, for generic $p\in B$ and any order $k$ we have $K^{\det}_{k,p}$ is a normal algebraic subgroup of ${\rm Gal}_k(\mathcal H_{\det})_p$. For the same reason we have that the kernel of $\pi_{\rm proj}$ is a differential algebraic group of the form
\begin{gather*}\ker(\pi_{\rm proj}) = K^{\rm proj} \times_B {\rm Id},\end{gather*}
where $K^{\rm proj}$ is a differential algebraic subgroup of ${\rm Gal}_\infty(\mathcal H_{\det})$. Moreover, for generic $p\in B$ and any order $k$ we have $K^{\rm proj}_{k,p}$ is a normal algebraic subgroup of ${\rm Gal}_k(\mathcal H_{\rm proj})_p$. The subgroup ${\rm Gal}_k(\mathcal H_{\rm})_p$ gives the graph of an algebraic group isomorphism
\begin{gather*}{\rm Gal}_k(\mathcal H_{\rm proj})_p/K^{\rm proj}_{k,p} \simeq
{\rm Gal}_k(\mathcal H_{\det})_p/K^{\det}_{k,p}.\end{gather*}
Therefore, the quotient ${\rm Gal}_k(\mathcal H_{\rm proj})_p/K^{\rm proj}_{k,p}$ is for any $k$ an abelian group of dimension smaller or equal than three. By Lemma~\ref{lm:simple_dsimple} the only possibility is ${\rm Gal}_k(\mathcal H_{\rm proj})_p = K^{\rm proj}_{k,p}$. This is for generic $p$ and all $k$, therefore ${\rm Gal}_\infty(\mathcal H_{\rm proj}) = K^{\rm proj}$ and it follows the required statement.
\end{proof}

The quotient map $\pi_{\rm red} \colon P \to P_{\rm red}$ induces a $2$-covering group bundle morphism
\begin{gather*} \pi_{\rm red*}\colon \ {\rm Gau}(\Gamma(\mathcal H))\to {\rm Gau}(\Gamma(\mathcal H_{\rm red})),\end{gather*} whose restriction to the Galois group bundles gives a surjective map
\begin{gather*} \pi_{\rm red}\colon \ {\rm Gal}_{\infty}(\mathcal H)\to {\rm Gal}_{\infty}(\mathcal H_{\rm red}).
\end{gather*} This morphism may be either an isomorphism or a $2$-covering. It turns out that $\pi_{\rm red}^{-1}({\rm Gal}_{\infty}(\mathcal H_{\rm red}))$ is connected, and the smallest differential algebraic subgroup of ${\rm Gau}(\Gamma(\mathcal H))$ that projects surjectively onto ${\rm Gal}_{\infty}(\mathcal H)_{\rm red}$. Therefore,
\begin{gather*}{\rm Gal}_{\infty}(\mathcal H) = \pi_{\rm red*}^{-1}({\rm Gal}_{\infty}(\mathcal H_{\rm red})),\nonumber\\
{\rm Gal}_{\infty}(\mathcal H) = \{j\sigma\in {\rm Gau}(\Gamma(\mathcal H)) \,|\, j {\det}(\sigma)\in {\rm Gal}_{\infty}(\mathcal H_{\det})\},
\end{gather*}
which is defined by the differential equations
\begin{gather*} \partial_x\sigma = \left[\left(\begin{matrix}
0 & 1 \\
\frac{\alpha\beta}{x(1-x)} & \frac{(\alpha+\beta+1)x - \gamma}{x(1-x)}
\end{matrix} \right),\sigma\right],\\
\partial_x\det(\sigma) = \partial_c\det(\sigma) = \partial_a\left(\frac{\partial_a\det(\sigma)}{\det(\sigma)}\right)=
\partial_b\left(\frac{\partial_a\det(\sigma)}{\det(\sigma)}\right)=
\partial_b\left(\frac{\partial_b\det(\sigma)}{\det(\sigma)}\right)= 0,
\end{gather*}
where $\sigma$ is an invertible $2\times 2$ matrix depending on $x$, $\alpha$, $\beta$, $\gamma$ and the differential operators $\partial_a$, $\partial_b$, $\partial_c$ are those given in Proposition~\ref{pr:abc}.

A consequence of the computation of ${\rm Gal}_{\infty}(\mathcal H)$ is the following: let us consider the Gauss hypergeometric series
\begin{gather*}{}_{2}{\rm F}_1(\alpha,\beta,\gamma;x) = \sum_{n=0}^\infty \frac{(\alpha)_n(\beta)_n}{(\gamma)_n}\frac{x^n}{n!}.\end{gather*}
It is well known that
\begin{gather*}{}_2{\bf F}_1 = \left(
\begin{matrix}
{}_2{\rm F}_1(\alpha,\beta,\gamma;x) & x^{1-\gamma}{}_2{\rm F}_1(1+\alpha-\gamma,1+\beta-\gamma,2-\gamma;x) \\
\dfrac{{\rm d}}{{\rm d}x}{}_2{\rm F}_1(\alpha,\beta,\gamma;x) & \dfrac{{\rm d}}{{\rm d}x}({}_2{\rm F}_1(1+\alpha-\gamma,1+\beta-\gamma,2-\gamma;x))
\end{matrix}
\right)\end{gather*}
is a fundamental matrix of solutions of~\eqref{HGmatrix} and thus $\mathcal H$-horizontal section. The only differential equation satisfied by~${}_{2}{\bf F}_1$ with respect to the parameters $\alpha$, $\beta$, $\gamma$ are those defining its Galois group, that are differential equations in~$\det({}_2{\bf F}_1)$. Therefore we may state that the \emph{hypergeometric series ${}_2{\rm F}_1(\alpha,\beta,\gamma; x)$ does not satisfy any algebraic partial differential equation with respect to the parameters~$\alpha$, $\beta$, $\gamma$.}

\appendix

\section{Differential algebraic groups}\label{ap:DAG}

Our treatment of differential algebraic groups is the geometrical theory proposed in \cite{le2010algebraic}. This formulation is slightly different from the better known algebraic formulation in \cite{kolchin1985}. Here we recall and comment the most important results, including Malgrage's version of Kiso--Cassidy theorem.

\subsection{Differential algebraic groups}\label{ss:DAG}

Let $\pi_{\mathscr G}\colon {\mathscr G}\to B$ a group bundle. As before, we assume ${\mathscr G}$ to be affine over $B$. The jet bundle $J_k(\mathscr G/B)\to B$ of $k$-jets of sections of $\pi_{\mathscr G}$ is an algebraic group bundle with the group operation $\big(j^k_xg\big)\big(j^k_xh\big) = j^k_x(gh)$. We consider $\mathcal O_{J_k(\mathscr G/B)}$ the ring of differential functions of order~$\leq k$ with rational coefficients in $B$. We may interpret it as the ring of regular functions in $J_k(\mathscr G/B)$ seen as an algebraic $\mathbb C(B)$-group. The group structure induces, by duality, a~Hopf $\mathbb C(B)$-algebra structure on~$\mathcal O_{J_K(\mathscr G/B)}$.

\begin{Remark}\label{Cartier}
Note that, by Cartier theorem (for instance, in \cite{demazure1970schemas}) $\mathbb C(B)$-algebraic groups are smooth (and therefore reduced), so that, Hopf ideals of $\mathcal O_{J_k(\mathscr G/B)}$ are radical.
\end{Remark}

\begin{Remark}
There is a natural bijective correspondence between Hopf ideals of $\mathcal O_{J_k(\mathscr G/B)}$ and rational group subbundles of $J_k(\mathscr G/B)\to B$. Given a Hopf ideal $\mathcal E$, the set of zeros of $\mathcal E\cap \mathbb C[J_k(\mathscr G/B)]$ is rational group subbundle. By abuse of notation we will refer to it as the set of zeros of $\mathcal E$.
\end{Remark}

The group multiplication is compatible with the order truncation, therefore we may take the projective limit, so that the jet bundle $J(\mathscr G/B)$ is a pro-algebraic group bundle and its ring of coordinates (on the generic point of $B$) $\mathcal O_{J(\mathscr G/B)} = \mathbb C(B)\otimes_{\mathbb C[B]} \mathbb C[J(\mathscr G/B)]$ a Hopf $\mathbb C(B)$-algebra.

Indeed, $\mathcal O_{J(\mathscr G/B)}$ is endowed with a differential ring structure, with the total derivative opera\-tors~$\mathfrak X(B)^{\rm tot}$. The Hopf and the differential structure are compatible in the following sense
\begin{gather*}\mu^*\big(X^{\rm tot}f\big) = \big(X^{\rm tot}\otimes 1 + 1 \otimes X^{\rm tot}\big)\mu^*(f),\qquad
e^*\big(X^{\rm tot}f\big) = X (e^*(f) ),\end{gather*}
where $\mu^*$ end $e^*$ stand for the coproduct and coidentity operators. Hence, we may say that $\mathcal O_{J(\mathscr G/B)}$ is a Hopf-differential algebra.

\begin{Definition}\quad
\begin{enumerate}\itemsep=0pt
\item A \emph{Hopf-differential} ideal of $\mathcal O_{J(\mathscr G/B)}$ is a Hopf ideal $\mathcal E$ of $\mathcal O_{J(\mathscr G/B)}$ such that $X^{\rm tot}\mathcal E \subset \mathcal E$ for every rational vector field~$X$ on~$B$.
\item A \emph{differential algebraic} subgroup $H\subset J(\mathscr G/B)$ is a Zariski closed subset consisting of the zero locus of some Hopf-differential ideal $\mathcal E$ of $\mathcal O_{J(\mathscr G/B)}$.
\end{enumerate}
\end{Definition}

As differential-Hopf ideals are radical (Remark \ref{Cartier}) we have that the natural correspondence between differential-Hopf ideals of $\mathcal O_{J(\mathscr G/B)}$
and differential algebraic subgroups of $J(G/B)$ is, in fact, bijective.

\begin{Example}\label{ex:KolchinG}Let $\mathscr G = {\rm GL}_n(\mathbb C) \times B$. A differential algebraic subgroup $H\subset J(\mathscr G/B)$ is a~differential algebraic group of $\mathscr G(\mathcal U)$ defined over $\mathbb C(B)$ in the sense of Kolchin (as in~\cite{kolchin1985}), where~$\mathcal U$ a universal differential field extension of $\mathbb C(B)$.
\end{Example}

Given a Hopf-differential ideal $\mathcal E$, its truncation $\mathcal E_k = \mathcal E\cap \mathcal O_{J_k(\mathscr G/B)}$ is Hopf-ideal of $\mathcal O_{J_k(\mathscr G/B)}$. Thus, a differential algebraic subgroup of $J(\mathscr G/B)$ is given by a chain of Hopf ideals
\begin{gather*}\mathcal E_0 \subset \mathcal E_1 \subset \cdots \subset \mathcal E_k \subset \cdots\subset \mathcal E\end{gather*}
and thus, a system of rational group subbundles of the $J_k(\mathscr G/B)$,
\begin{gather*}H\to \cdots \to H_k \to \cdots \to H_1 \to H_0 \to B.\end{gather*}

The differential ring $\mathcal O_{J(\mathscr G/B)}$ is, by construction, the quotient of a ring of differential polynomials with coefficients in $\mathbb C(B)$. Thus, it has the \emph{Ritt property} (see, for instance, \cite{kolchin1973}): radical differential ideals are finitely generated as radical differential ideals. Thus, there is a smallest order $k$ such that $\mathcal E$ coincides with $\{\mathcal E_k\}$, the radical differential ideal spanned by $\mathcal E_k$. We have that $H$ is completely determined by $H_k$ and we say that $H$ is a differential algebraic subgroup of $J(\mathscr G/B)$ of \emph{order} $k$.

A \emph{section} of $H$ is a section $g$ of $\pi_{\mathscr G}\colon \mathscr G\to B$ such that $j_xg\in H$ for any $x$ in the domain of definition of $g$. We can deal with rational, local analytic, convergent, or formal sections of $H$. We can also speak of sections of $H$ with coefficients in any differential field extension of $\mathbb C(B)$.

\begin{Example}The points of $H$ in the sense of Kolchin (as in Example \ref{ex:KolchinG}) are sections of $H$ with coefficients in arbitrary differential field extensions of $\mathbb C(B)$. In that context, it is useful to fix a differentially closed field extension of $\mathbb C(B)$ so that $H$ is determined by its sections.
\end{Example}

\subsection{Differential Lie algebras}\label{ss:DAL}

Let us consider $\pi_{\mathscr G}\colon {\mathscr G}\to B$ as before; let us define $\mathfrak g = e^*(\ker({\rm d}\pi_{\mathscr G}))$, therefore $\pi_{\mathfrak g} \colon \mathfrak g \to B$ is the Lie algebra bundle associated to ${\mathscr G}$.

In an analogous way the jet bundles of finite order $J_k(\mathfrak g/B)\to B$ are Lie algebra bundles and the limit $J(\mathfrak g/B)\to B$ is a pro-algebraic linear bundle of Lie algebras; the Lie bracket is defined as $[j_x v_1, j_xv_2] = j_x[v_1,v_2]$ for jets of local analytic sections.

We consider its ring of coordinates on the generic point of $B$, $\mathcal O_{J(\mathfrak g/B)}= \mathbb C(B)\otimes_{\mathbb C[B]}\mathbb C[J(\mathfrak g/B)]$. It has a Lie coalgebra structure with the induced comultiplication
\begin{gather*}\beta^*\colon \ \mathcal O_{J(\mathfrak g/B)} \to \mathcal O_{J(\mathfrak g/B)}\otimes_{\mathbb C(B)} \mathcal O_{J(\mathfrak g/B)}, \qquad
b^*(f)(j_xv_1,j_xv_2) = f([j_xv_1,j_xv_2]).\end{gather*}

Moreover, $\mathcal O_{J(\mathfrak g/B)}$ is a differential ring with the differential structure induced by total derivative operators $\mathfrak X(B)^{\rm tot}$. Indeed, the linear bundle structure of $J(\mathfrak g/B)$ is reflected in the ring $\mathcal O_{J(\mathfrak g/B)}$: linear differential functions are identified with linear
differential operators from the space of sections of $\mathfrak g$ on $\mathbb C(B)$. We have a sub-$\mathfrak X(B)$-module ${\rm Diff}_B(\mathfrak g)\subset \mathcal O_{J(\mathfrak g/B)}$ that spans the ring of differential functions (as in \cite[Section~4]{le2010algebraic}).

\begin{Definition}A differential Lie subalgebra $\mathfrak h\subset J(\mathfrak g/B)$ is a Zariski closed subset consisting of the zeroes of a differential ideal~$\mathcal E$ of $\mathcal O_{J(\mathfrak g/B)}$ satisfying:
\begin{enumerate}\itemsep=0pt
\item[1)] $\mathcal E$ is spanned by its linear part $\mathcal E \cap {\rm Diff}_B(\mathfrak g)$;
\item[2)] $\mathcal E$ is a Lie co-ideal.
\end{enumerate}
\end{Definition}

Note that a differential Lie subalgebra is also a differential subgroup with respect to the additive structure of $\mathfrak g$. There is an smallest $k$ such that $\{\mathcal E_k\} = \mathcal E$, the \emph{order} of~$\mathfrak h$; and $\mathfrak h$ is determined by its truncation $\mathfrak h_k$.

A differential algebraic group $H$ has an associated differential Lie algebra $\mathfrak h$ of the same order. For each $\ell$, the Lie algebra bundle~$\mathfrak h_{\ell}\to B$ is (on the generic point of $B$) the Lie algebra of the group bundle $H_{\ell}\to B$. Equations of $\mathfrak h$ can be found by linearization of the equations of~$H$ along the identity (see \cite[Section~4.5]{le2010algebraic}).

\subsection{Group connections}\label{ss:GC}
Let us consider a regular connection $\mathscr C$ on $\pi_{\mathscr G}\colon \mathscr G\to B$, that is, a regular foliation in $\mathscr G$ transversal to the fibers of $\pi_{\mathscr G}$. For each $g\in \mathscr G$ the leaf of $\mathscr C$ that passes through $g$ is the graph of a local analytic section $\tilde g$ of $\pi_{\mathscr G}$. The connection $\mathscr C$ defines a regular map
\begin{gather*}\gamma_{\mathscr C}\colon \ \mathscr G \to J(\mathscr G/B), \qquad g\mapsto j_{\pi_{\mathscr G}(g)}\tilde g,\end{gather*}
that is a section of $J(\mathscr G/B)\to \mathscr G$ and identifies $\mathscr G$ with the set of jets of $\mathscr C$-horizontal sections~$\Gamma(\mathscr C)$.

\begin{Definition}We say that $\mathscr C$ is a \emph{group connection} if $\gamma_{\mathscr C}$ is a group bundle morphism.
\end{Definition}

In general, a \emph{rational group connection} is a regular group connection defined on some dense Zariski open subset of $B$. Given a rational group connection $\mathscr C$, the variety $\Gamma(\mathscr C)$ of $\mathscr C$-horizontal leaves is an order $1$ differential algebraic subgroup of $J(\mathscr G/B)$ isomorphic (as group bundle, on the generic point of $B$) to $\mathscr G$.

\begin{Remark}Differential algebraic groups of the form $\Gamma(\mathscr C)$, for a group connection $\mathscr C$, correspond to \emph{constant} differential algebraic groups in the sense of Kolchin \cite{kolchin1985}. The group connection is the differential equation of the conjugation of the differential algebraic group with an algebraic group over the constants.
\end{Remark}

Analogously, a regular \emph{linear connection} $\mathscr L$ in $\mathfrak g\to B$ induces a~section $\gamma_{\mathscr L}$ of $J(\mathfrak g/B)\to \mathfrak g$ which is linear in fibers. We say that $\mathscr L$ it is a \emph{Lie connection} if $\gamma_{\mathscr L}$ is a Lie algebra morphism in fibers. A rational Lie connection is a regular Lie connection in some Zariski open subset of $B$.
The variety of horizontal sections $\Gamma(\mathscr L)$ of a rational Lie connection is a differential Lie subalgebra of $J(\mathfrak g/B)$ of order $1$.

It is clear that a group connection $\mathscr C$ can be linearized along the identity, producing Lie connection $\mathscr C'$ with the same domain of regularity. In such case, $\Gamma(\mathscr C')$ is the Lie algebra of $\Gamma(\mathscr C)$.

We may also define partial group and Lie connections. Let us fix $\mathscr F$ a singular foliation in $B$. Let us recall that a $\mathscr F$-connection in $\mathscr G$ (resp. $\mathfrak g$) is just an foliation $\mathscr C$ of the same rank of $\mathscr F$ and that projects onto $\mathscr F$. A local section is $\mathscr C$-horizontal if its derivative maps $\mathscr F$ to $\mathscr C$. The jets of local sections of $\mathscr C$ define a closed subset ${\Gamma}(\mathscr C)$ in $J(\mathscr G/B)$ (resp.\ in $J(\mathfrak g/B)$)
\begin{gather*}{\Gamma}(\mathscr C) = \overline{\{j_x s \,|\, x\in B,\, s \mbox{ local }\mathscr C\mbox{-horizontal section} \} }.\end{gather*}

A \emph{partial group $\mathscr F$-connection} is a partial Ehresmann $\mathscr F$-connection $\mathscr C$ whose set of horizontal sections ${\Gamma}(\mathscr C) \subset J(\mathscr G/B)$
is a differential algebraic group. A \emph{partial Lie $\mathscr F$-connection} is a $\mathscr F$-connection whose set of horizontal sections is a differential Lie subalgebra. As before, those objects are of order 1, and the linearization of a group $\mathscr F$-connection is a Lie $\mathscr F$-connection.

We say that a differential subgroup $H\subset J(\mathscr G/B)$ is \emph{Zariski dense} in $\mathscr G$ if $H_0 = \mathscr G$. In the algebraic setting, the classification of Zariski dense subgroups in simple groups is due to Cassidy~\cite{cassidy1989classification}. We follow the exposition of Malgrange, that finds the key argument of the proof in the a~result of Kiso~\cite{kiso1979local}.

\begin{Theorem}[{Kiso--Cassidy, \cite[Theorem~8.1]{le2010algebraic}}]\label{th:KC} Let $\mathfrak g\to B$ be a Lie algebra bundle with simple fibers. Let $\mathfrak h\subset J(\mathfrak g/B)$ be a differential subalgebra such that $\mathfrak h_0 = \mathfrak g$. Then there is a~singular foliation $\mathscr F$ in $B$ and a~partial Lie $\mathscr F$-connection $\mathscr L$ in $\mathfrak g$ such that $\mathfrak h = \Gamma(\mathscr L)$.
\end{Theorem}

\begin{Corollary}Let $\mathscr G\to B$ be a group bundle with simple fibers. Let $H\subset J(\mathfrak g/B)$ be a~differential algebraic subgroup such that $H_0 = \mathscr G$. Then there is a singular foliation $\mathscr F$ in $B$ and a~partial group $\mathscr F$-connection~$\mathscr C$ in $\mathcal G$ such that $H = \Gamma(\mathscr C)$.
\end{Corollary}

\subsection{The adjoint connection}\label{ss:Adjoint_cnx}

Let us consider now $\pi\colon P\to B$ a principal bundle with structural group $G$. Let $\mathscr F$ be a singular foliation on $B$ and $\mathscr D$ be a partial $G$-invariant $\mathscr F$-connection on $P$. Let~$F$ be an affine $G$-space with an action $\alpha\colon G\to {\rm Aut}(F)$. The associated bundle associated bundle~$P[\alpha]$ is defined as the balanced construction $P\times_{\alpha} F$; it is a bundle $P[\alpha]\to B$ with fibers $F$. The $G$-invariant foliation $\mathscr D \times \{0\}$ in $P\times F$ projects onto a foliation on $P[\alpha]$ that we call ${\alpha(\mathscr D)}$; it is the induced partial Ereshmann $\mathscr F$-connection on $P[\alpha]$.

Let ${\rm Gau}(P)\to B$ be the group bundle of gauge automorphisms of $P$. Let us recall that ${\rm Gau}(P)$ can be constructed as the associated bundle $P[{\rm Adj}] = (P\times_{\rm Adj} G)$ for the adjoint action of~$G$ on itself. The induced $\mathscr F$-connection ${\rm Adj}(\mathscr D)$ turns out to be a group $\mathscr F$-connection in~${\rm Gau}(P)$.

\begin{Theorem}[{\cite[Theorem~7.6]{le2010algebraic}}]\label{th:Mal} Assume that $G$ is semi-simple and connected, and let~$\mathscr F$ a~singular foliation in $B$. Then all partial group $\mathscr F$-connections in ${\rm Gau}(P)$ are of the form ${\rm Adj}(\mathscr D)$ where~$\mathscr D$ is a~partial $G$-invariant $\mathscr F$-connection in $P$.
\end{Theorem}

\begin{Remark}Note that, since the adjoint action of semi-simple $G$ in itself is faithful, the $G$-invariant partial connection $\mathscr D$ inducing ${\rm Adj}(\mathscr D)$ is unique.
\end{Remark}

\begin{Example} Let us consider compatible linear equations in $P = {\rm GL}_n(\mathbb C) \times \mathbb C^{n+m}$:
\begin{gather*}\frac{\partial U}{\partial x_i} = A_i(x,s)U.\end{gather*}
Then ${\rm Gau}(P) = {\rm GL}_n(\mathbb C) \times \mathbb C^{n+m}$ and the induced group connection is
\begin{gather*}\frac{\partial \sigma}{\partial x_i} = [A_i(x,s),\sigma].\end{gather*}
The horizontal sections $\sigma(x,s)$ are gauge symmetries of the equation.
\end{Example}

Theorems \ref{th:KC} and~\ref{th:Mal} give us a complete description of \emph{Zariski dense} differential algebraic subgroups of $J({\rm Gau}(P)/B)$ when $G$ is simple and connected. Such subgroups are of the form $\Gamma({\rm Adj}(\mathscr D))$ for a partial $G$-invariant connections~$\mathscr D$ in~$P$.

\section{Parametrized Picard--Vessiot theory}\label{appendix-B}

Here we discuss the relation between the differential Galois theory for connections with parameters, and the parameterized Picard--Vessiot theory presented in \cite{cassidy2005galois}.

Let us consider a rational (in practice, we assume regular) partial ${\rm GL}_n(\mathbb C)$-invariant connection $\mathscr D$ on $P = \rm{GL}_n(\mathbb C)\times B$ parameterized by $\rho\colon B\to S$, with $\mathbb C(S)$ relatively algebraically closed in $\mathbb C(B)$. We consider the following basis of $\mathfrak X(B)$: we take a trancendence basis $\{s_1,\ldots,s_m\}$ of $\mathbb C(S)$ and then extend it to a transcendence basis $\{s_1,\ldots,s_m,x_1,\ldots,x_m\}$; then we consider $\Pi = \big\{\frac{\partial}{\partial s_i}\big\}$, $\Lambda = \big\{\frac{\partial}{\partial x_i} \big\}$ and $\Delta = \Pi\cup\Lambda$. Then $\mathbb C(B)$ is a $\Delta$-field, and $\mathbb C(B)^{\Lambda} = \mathbb C(S)$. The differential equation for the $\mathscr D$-horizontal sections is a linear integrable system
\begin{gather}\label{eq:linearS}
\frac{\partial U}{\partial x_i} = A_iU, \qquad A_i\in\mathfrak{gl}_n(\mathbb C(B)),
\end{gather}
satisfying compatibility conditions $\frac{\partial A_j}{\partial x_i} - \frac{\partial A_i}{\partial x_j} = [A_i,A_j]$.

\begin{Definition}[{\cite[Definition~9.4(1)]{cassidy2005galois}}] A PPV ring over $\mathbb C(B)$ for the equation~\eqref{eq:linearS} is a~$\Delta$-ring~$R$ containing $\mathbb C(B)$ and satisfying:
\begin{itemize}\itemsep=0pt
\item[(a)] $R$ is a $\Delta$-simple $\Delta$-ring (i.e., it does not contain non-trivial $\Delta$-ideals);
\item[(b)] there exist a matrix $\Phi\in {\rm GL}(R)$ such that $\frac{\partial \Phi}{\partial x_i} = A_i\Phi$;
\item[(c)] $R$ is spanned, as $\Delta$-ring by the entries of $\Phi$ and the inverse of its determinant; $R = \mathbb C(B)\big\{\Phi,\Phi^{-1}\big\}_{\Delta}$.
\end{itemize}
\end{Definition}

Let us recall the construction of the differential variety $\Gamma(\mathscr D) \subset J(P/B)$. It is the variety of jets of horizontal sections, so that
\begin{gather*}\mathcal O_{\Gamma(\mathscr D)} = \mathbb C(B)\big\{u_{ij},\det(u_{ij})^{-1}\big\}_\Delta / \mathcal E,\end{gather*}
where $\mathcal E$ is the radical differential ideal spanned by the connection $\mathcal E =\big\{\frac{\partial u_{ij}}{\partial x_k} - a_{ik}u_{kj} \big\}$. Therefore, by construction, $\mathcal O_{\Gamma(\mathscr D)}$ satisfies conditions~(a) and~(c) of the definition, but it is not in general $\Delta$-simple. However, the following is clear.

\begin{Proposition}For any maximal $\Delta$-ideal $\mathfrak m\subset\mathcal O_{\Gamma(\mathscr D)}$ the quotient $R_{\mathfrak m} = \mathcal O_{\Gamma(\mathscr D)}/\mathfrak m$ is a PPV ring over $\mathbb C(B)$ for equation~\eqref{eq:linearS}.
\end{Proposition}

\begin{Definition}[{\cite[Definition~9.4(2)]{cassidy2005galois}}] A PPV field over $\mathbb C(B)$ for the equation~\eqref{eq:linearS} is a~$\Delta$-field~$R$ containing $\mathbb C(B)$ and satisfying:
\begin{itemize}\itemsep=0pt
\item[(a)] there exist a matrix $\Phi\in {\rm GL}(L)$ such that $\frac{\partial \Phi}{\partial x_i} = A_i\Phi$;
\item[(b)] $L$ is spanned, as $\Delta$-ring by the entries of $\Phi$, $L = \mathbb C(B)\big\langle\Phi,\Phi^{-1}\big\rangle_{\Delta}$;
\item[(c)] $L^{\Lambda} = \mathbb C(B)^{\Lambda}$.
\end{itemize}
\end{Definition}

A big problem is that PPV rings does not always produce PPV fields, and we do not have in general an uniqueness theorem, unless the field of $\Lambda$-constants is constrainedly closed (which it is not the case in our scope).

However, let us assume that $\rho\circ \pi\colon P\to S$ admit a section $\tau$. Since $P$ is a direct product $\tau = (\phi,\bar\tau)$, so that $\bar\tau$ is a section of~$\rho$. Then, we can consider the section $\tau$ as an initial value condition for \eqref{eq:linearS}, so that $\tau$ prolongs to a section $j\tau\colon S\to \Gamma(\mathscr D)$ that maps each $s\in S$ to $j_{\pi(\tau(s))}\Phi$ where $\Phi$ is the solution of the initial value problem
\begin{gather*}\frac{\partial \Phi}{\partial x_i} = A_i\Phi, \qquad \Phi(\bar\tau(s)) = \phi(s).\end{gather*}
Now, we consider $T_{\tau} = V(j\tau(S),\mathscr D^{(\infty)})$ the smallest subvariety variety of $\Gamma(\mathscr D)$ that contains the image of $j\tau$ and it is differential (tangent to $\mathscr D^{\infty}$). This $T_\tau$ (our candidate for torsor) is by construction the zero locus of a maximal differential ideal $\mathfrak m_{\tau}$.

\begin{Proposition}The $\Delta$-field extension $\mathbb C(B)\subset \mathbb C(T_{\tau})$ is a PPV field over $\mathbb C(B)$ for equa\-tion~\eqref{eq:linearS}.
\end{Proposition}

\begin{proof}
A first observation is that the PPV ring $\mathcal O_{T_{\tau}}$ is not only $\Delta$-simple but also $\Lambda$-simple.

Let us note that for each $k$, the truncation $T_{\tau,k}$ is the smallest subvariety tangent to $\mathscr D^{(k)}$ that contains the image of $j_k\tau$.

If $\mathfrak m$ is a non-trivial maximal $\Lambda$-ideal of $\mathcal O_{T_{\tau}}$ then its zero locus is a subvariety $T'\subset T_{\tau}$ strictly contained in $T_{\tau}$ that dominates $B$ by projection. We may take an algebraic section $\tau'$ from $S$ to $T'$ such that $\pi_{0}\circ\tau' = \pi\circ \tau$, where $\pi_0$ is the projection from $J(P/B)$ to $B$. Then, we have an algebraic map $g\colon S\dasharrow \hat\tau_m G$ such that $j\tau\cdot g = \tau'$. Then, $T'\cdot g^{-1}$ contains the image of $j\tau$ and it is strictly contained in $T_\tau$.

Therefore, for any $k$ we have that $T'_k$ is a subvariety of $T_{\tau,k}$ containing the image of $j_k\tau$. Since the ideal of $T'_k$ is a $\Lambda$-ideal, we have that $\big(T'\cdot g^{-1}\big)_k$ is tangent to $\mathscr D^{(k)}$. Therefore, by the minimality of $T_{\tau,k}$ we have that $T'_k \big(T'\cdot g^{-1}\big) = T_{\tau,k}$ for all~$k$, so that $\big(T'\cdot g^{-1}\big) = T_{\tau}$, and then $T' = T_{\tau}$.

Let $F\in\mathbb C(T_\tau)$ be a $\Lambda$-constant. It defines a non-zero $\Lambda$-ideal of $\mathcal O_{T_\tau}$,
\begin{gather*}\{b\in \mathcal O_{T_\tau} \colon bF \in \mathcal O_{T_\tau} \}.\end{gather*}
This ideal must be $(1)$ and then we have $F\in \mathcal O_{T_\tau}$. We have that $j\tau^*\colon \mathcal O_{T_\tau}\to \mathbb C(S)$ is a retract of the embedding of $\mathbb C(S)$ into $\mathcal O_{T_\tau}$, and it is a $\Pi$-ring morphism. Then, $F - j\tau^*(F)$ is a $\Lambda$-constant, and then it is also a $\Pi$-constant, so that it is a~$\Delta$-constant, and therefore a complex number. Moreover $F$ and $j\tau^*(F)$ coincide along the image of $j(\tau)$ so we have that $F = j\tau^*(F)$, a rational function in~$S$.
\end{proof}

\begin{Remark}Note that $\pi$ admit always an algebraic section, therefore, in our scope, there is always a Picard--Vessiot extension after an algebraic extension of $\mathbb C(S)$. The existence of such extensions was proved in a much more general situation in~\cite{wibmer2012}.
\end{Remark}

In $\Gamma(\mathscr D)$ we have the following situation. The group bundle ${\rm Gau}(\Gamma(\mathscr D))$ of gauge symmetries of $\mathscr D$, and the group $J({\rm GL}_n(\mathbb C)/S)$ of right translations depending on parameters acts on the right side. These actions commute,
\begin{gather*}{\rm Gau}(\Gamma(\mathscr D))\times_B \Gamma(\mathscr D) \times_S J({\rm GL}_n(\mathbb C)/S)\to \Gamma(\mathscr D).\end{gather*}

In order to understand the relation between the Galois group bundle ${\rm Gal}_\infty(\mathscr D)$ and the parameterized Picard--Vessiot group of~\cite{cassidy2005galois} let us first examine what happens in case without parameters, where $S$ reduces to a point. In such case $\tau$ is just a point $p\in P$. By G\'omez-Mont theorem (Theorem~\ref{Gomez-Mont_as_ROE}) we may consider a complete integral $f\colon P\dasharrow V$ of $\mathscr D$ with $V$ a model for $\mathbb C(P)^\mathscr D$ and we have that $T = \overline{f^{-1}f(p)}$.

Let us consider ${\rm Stb}(T)$ the algebraic subgroup of $G$ that stabilizes $T$ by right translations. It acts in $\mathbb C(T)$ by $\mathbb C(B)$-automorphisms on the left side, and it is the differential Galois group ${\rm Aut}_\Lambda(\mathbb C(T)/\mathbb C(B))$ of classical Picard--Vessiot theory (without parameters).

The action of the Galois bundle ${\rm Gal}(\mathscr D)$ respects the values of rational first integrals of $\mathscr D$. Therefore we have ${\rm Gal}(\mathscr D)\times_B T \times {\rm Stb}(T)\to T$ a restricted action. The fibers of~$T$ are principal homogeneous spaces, on the right side for~${\rm Stb}(T)$ and on the left side for the fibers of ${\rm Gal}(\mathscr D)$. If we fix the point $x = \pi(p)$ we have a~group isomorphism
\begin{gather*}\psi\colon \ {\rm Gal}(\mathscr D)|_{x} \xrightarrow{\sim} {\rm Stb}(T), \qquad g_x\mapsto \psi(g_x) \qquad \mbox{with} \quad g_xp = p\psi(g_x).\end{gather*}

The case with parameters $S$ and a section $\tau$ is not so different. As $T_{\tau}$ is the zero locus of a~$\Delta$-ideal, we have that the stabilized of $T_{\tau}$ is a~differential algebraic subgroup of $J({\rm GL}_n(\mathbb C/S)$.

Also, by iterated application of G\'omez-Mont theorem we may construct a complete integral $f_{\infty}\colon \Gamma(\mathscr D)\dasharrow V_{\infty}$ of $\mathscr D^{(\infty)}$ where $V_\infty$ is a model of $\mathbb C(\Gamma(\mathscr D))^{\mathscr D^{(\infty)}}$. We have a commutative diagram
\begin{gather*}\xymatrix{\Gamma(\mathscr D) \ar@{-->}[r]^-{f_\infty} \ar[d] &
V_\infty \ar@{-->}[dl]\\ S, \ar@/^/[u]^-{j\tau} }\end{gather*}
and $T_{\tau} = \overline{f_\infty^{-1}(f_{\infty}(j\tau(S)))}$. Be definition, the action of ${\rm Gal}_\infty(\mathscr D)$ preserves $f_\infty$ and therefore it stabilizes $T_{\tau}$. We have again left and right actions
\begin{gather*}{\rm Gal}_{\infty}(\mathscr D)\times_B T_{\tau} \times_S {\rm Stb}(T_\tau) \to T_{\tau},\end{gather*}
such that any fiber of $T_\tau$ is a principal homogeneous space for the corresponding fibers of ${\rm Gal}_{\infty}(\mathscr D)$ and ${\rm Stb}(T_\tau)$.

The relation between the differential algebraic group ${\rm Stb}(T_\tau)\to S$ and the automorphisms of $\mathbb C(T_\tau)$ is as follows. A \emph{rational section} of ${\rm Stb}(T_\tau)$ is a rational section $g$ of ${\rm Stb}_0(T_{\tau})$ (the $0$-th order truncations) whose jet prolongation $jg$ is a section of ${\rm Stb}(T_\tau)$.
Rational sections of ${\rm Stb}(T_\tau)$ act in $\mathbb C(T_\tau)$ by $\mathbb C(B)$-automorphisms. Indeed, if we consider an universal $\Delta$-field~$\mathcal U$ containing~$\mathbb C(T_\tau)$ then sections of~${\rm Stb}(T)$ with coefficients in some $\Pi$-field $\mathcal K\subset \mathcal U^\Lambda$ yield $\mathbb C(B)$-isomorphisms of~$\mathbb C(T_\tau)$ into $\mathbb C(B)\cdot \mathcal K$.

As before, we can restrict ${\rm Gal}_{\infty}(\mathscr D)$, not to a point of $B$ but along the section $\bar\tau$ of $\rho$, so that $\bar\tau^*({\rm Gal}_\infty(\mathscr D) )\to S$ is a differential algebraic group over $S$. We have an isomomorphism of differential algebraic groups
\begin{gather*}\psi\colon \ \bar\tau^*({\rm Gal}_\infty(\mathscr D))\xrightarrow{\sim} {\rm Stb}(T_\tau), \qquad
\sigma_{\bar\tau(s)} \mapsto g_s \qquad \mbox{with} \quad \sigma_{\bar \tau(s)}j\tau(s) = j\tau(s)g_s.\end{gather*}

\begin{Remark}Note that $\bar\tau^*$ kills the derivatives in $\Lambda$, therefore if we are interested in presenting the equations of the computed Galois groups in the setting of \cite{cassidy2005galois}, as differential algebraic groups over~$S$, we only need to take~$x$ constant and omit the equations of $\Gamma({\rm Adj}(\mathscr D))$ (which are transversal to $j\tau$) involving the derivative with respect to~$x$.
\end{Remark}

\subsection*{Acknowledgements}

We would thank the Universit\'e de Rennes 1 and the Universidad Nacional de Colombia for their hospitality and support. D.~Bl\'azquez-Sanz is partially funded by Colciencias project ``Estructuras lineales en geometr\'ia y topolog\'ia'' 776-2017 code 57708 (Hermes UN 38300). G.~Casale is partially funded by Math-AMSUD project ``Complex Geometry and Foliations''. J.S.~D\'iaz Arboleda is partially funded by Colciencias program 647 ``Doctorados Nacionales''. D.~Bl\'azquez-Sanz specially thanks the support, care, and patience of D.~Higuita during the writing process.

\pdfbookmark[1]{References}{ref}
\LastPageEnding

\end{document}